\documentclass[article,oneside,oldfontcommands,a4paper,11pt]{memoir}

\counterwithout{section}{chapter}
\setsecnumdepth{subsubsection}

\setsecheadstyle{\large\bf\raggedright}
\setsubsecheadstyle{\large\sl\raggedright}
\setsubsubsecheadstyle{\sl\raggedright}
\setsecnumformat{\csname the#1\endcsname.\;}

\usepackage[affil-it]{authblk}

\usepackage{indentfirst}
\usepackage[margin=1.05 in]{geometry}

\usepackage[numbers,sort&compress]{natbib}

\usepackage{amssymb,amsmath,mathrsfs,dsfont}
\usepackage{amsthm,stmaryrd}
\usepackage[hypertexnames=false,colorlinks,linkcolor={blue},citecolor={blue}]{hyperref}
\usepackage{graphicx,float,caption}
\usepackage{enumitem}

\usepackage[T1]{fontenc}

 \usepackage[usenames,dvipsnames]{pstricks}
 \usepackage{epsfig}
 \usepackage{pst-grad} 
 \usepackage{pst-plot} 
 \usepackage[space]{grffile} 
 \usepackage{etoolbox} 
 \makeatletter 
 \patchcmd\Gread@eps{\@inputcheck#1 }{\@inputcheck"#1"\relax}{}{}
 \makeatother

\usepackage{fancyhdr}
\numberwithin{equation}{section}

\newtheorem{thm}{Theorem}[section]

\newtheorem{lem}[thm]{Lemma}
\newtheorem{prop}[thm]{Proposition}
\theoremstyle{definition}
\newtheorem{dfn}[thm]{Definition}

\theoremstyle{remark}
\newtheorem*{rmk}{Remark}

\allowdisplaybreaks

\newcommand{\reDeclareMathOperator}[2]{\let#1\undefined \DeclareMathOperator{#1}{#2}}

\reDeclareMathOperator{\mod}{mod}

\DeclareMathOperator{\dom}{dom}

\reDeclareMathOperator{\supp}{supp}
\DeclareMathOperator*{\argmin}{arg\,min}
\DeclareMathOperator{\prox}{prox}
\reDeclareMathOperator{\Proj}{Proj}
\DeclareMathOperator{\gr}{gr}

\DeclareMathOperator{\Fix}{Fix}

\renewcommand{\emptyset}{\varnothing}
\newcommand{\CAT}{\rm CAT}
\newcommand{\0}{\mathbf{0}}
\newcommand{\R}{\mathbb{R}}
\newcommand{\N}{\mathbb{N}}

\renewcommand{\to}{\rightarrow}

\newcommand{\tendsto}{\longrightarrow}
\newcommand{\norm}[1]{\|#1\|}

\newcommand{\abs}[1]{\left|#1\right|}
\newcommand{\product}[2]{\langle#1,#2\rangle}

\renewcommand{\multimap}{\rightrightarrows}

\renewcommand{\angle}{\measuredangle}
\renewcommand{\tilde}{\widetilde}

\renewcommand{\vec}[1]{\overrightarrow{#1}}

\makeatletter
\renewcommand*{\@fnsymbol}[1]{\ifcase#1\or*\else\@arabic{#1}\fi}
\makeatother

\title{Monotone vector fields and generation of nonexpansive semigroups in complete $\rm CAT(0)$ spaces}

\author[1]{Parin Chaipunya\thanks{Corresponding author.}$^{,}$}
\author[2]{Fumiaki Kohsaka}
\author[1]{Poom Kumam}

\affil[1]{
Department of Mathematics, Faculty of Science,\break
King Mongkut's University of Technology Thonburi, \break
126 Pracha Uthit Rd., Bang Mod, Thung Khru,\break
Bangkok 10140, Thailand.\break
Emails: parin.cha@mail.kmutt.ac.th (P. Chaipunya),\break poom.kum@kmutt.ac.th (P. Kumam).
\vspace{.25cm}
}

\affil[2]{
Department of Mathematical Science,\break
Tokai University,\break
4-1-1 Kitakaname, Hiratsuka,\break
Kanagawa 259-1292, Japan.\break
Email: f-kohsaka@tsc.u-tokai.ac.jp (F. Kohsaka).
}

\date{}
\begin{document}
\maketitle \vspace{-1.2cm}
\thispagestyle{empty}
\begin{abstract}

In this paper, we discuss about monotone vector fields, which is a typical extension to the theory of convex functions, by exploiting the tangent space structure. This new approach to monotonicity in $\CAT(0)$ spaces stands in opposed to the monotonicity defined earlier in $\CAT(0)$ spaces by Khatibzadeh and Ranjbar \cite{MR3679017} and Chaipunya and Kumam \cite{MR3691338}. In particular, this new concept extends the theory from both Hilbert spaces and Hadamard manifolds, while the known concept barely has any obvious relationship to the theory in Hadamard manifolds. We also study the corresponding resolvents and Yosida approximations of a given monotone vector field and derive many of their important properties. Finally, we prove a generation theorem by showing convergence of an exponential formula applied to resolvents of a monotone vector field. Our findings improve several known results in the literature including generation theorems of Jost \cite[Theorem 1.3.13]{MR1652278}, Mayer \cite[Theorem 1.13]{MR1651416}, Stojkovic \cite[Theorem 2.18]{MR2879568}, and Ba\v{c}\'{a}k \cite[Theorem 1.5]{MR3047087} for proper, convex, lower semicontinuous functions in the context of complete $\CAT(0)$ spaces, and also by Iwamiya and Okochi \cite[Theorem 4.1]{MR1979730} for monotone vector fields in the context of Hadamard manifolds.

\medskip

\noindent{\bf Keywords:} Monotone vector field, Nonexpansive semigroup, Resolvent, Yosida approximation, Tangent space, $\CAT(0)$ space.

\noindent{\bf 2010 MSC:} 90C33, 65K15, 49J40, 49M30, 47H05.
\end{abstract}


\section{Introduction}


Convex functions in $\CAT(0)$ spaces were first considered during the 1990s. In particular, Jost \cite{MR1360608} and Mayer \cite{MR1651416} independently studied the proximal operators of convex functions in complete $\CAT(0)$ and applied them to investigate harmonic functions and also gradient flows. The proximal operators were later used by Ba\v{c}\'ak \cite{MR3047087} to study proximal algorithms for minimizing convex functions. He also proved a semigroup generation theorem from the exponential formula applied to these proximal operators. Further explanations can also be found in \cite{MR3241330}. In 2017, Khatibzadeh and Ranjbar \cite{MR3679017} as well as Chaipunya and Kumam \cite{MR3691338} generalized the results of Ba\v{c}\'ak \cite{MR3047087} by introducing the monotone operators and examined proximal algorithms using the dual space concept introduced earlier by Ahmadi Kakavandi and Amini \cite{MR2680038} (also see \cite{MR3003694}). The dual space used in \cite{MR2680038} is known to generalize the usual dual space when the linear structure is provided. However, it is still unknown (as also posted in \cite{MR2680038}) what relationship this dual space has with the Riemannian tangent spaces when the space in question is an Hadamard manifold. Consequently, the monotone operators introduced in this aspect barely have obvious or simple relationships with monotone vector fields on Hadamard manifolds as introduced by N\'emeth \cite{MR1694468} and later enriched by Li et al. \cite{MR2506692}.

In this present paper, we adopt a new approach to generalize and unify the concept of monotonicity into complete $\CAT(0)$ spaces and ultimately prove a semigroup generation theorem. Instead of applying the dual space of \cite{MR2680038}, we make use of the tangent spaces (also called tangent cones) of Nikolaev \cite{MR1317739} which is consistent with the contexts of Hilbert spaces or Hadamard manifolds. With this nice attribute of tangent spaces, the concept of monotone vector fields introduced in this paper is a natural complement to the theory in Hilbert spaces as well as in Hadamard manifolds. In addition to the main results, we introduce and deduce several fundamental properties of resolvents and Yosida approximations which become the major machinery in proving the generation theorem and are also important in their own rights. 

Note that the convergence procedure applied in our generation theorem (Theorem \ref{thm:S(t)x}) was used by Jost \cite[Theorem 1.3.13]{MR1652278}, Mayer \cite[Theorem 1.13]{MR1651416}, Stojkovic \cite[Theorem 2.18]{MR2879568}, and Ba\v{c}\'{a}k \cite[Theorem 1.5]{MR3047087} for proper, convex, lower semicontinuous functions in the context of complete $\CAT(0)$ spaces, and also by Iwamiya and Okochi \cite[Theorem 4.1]{MR1979730} for monotone vector fields in the context of Hilbert manifolds. Our generation theorem elevates the abovementioned results of \cite{MR1652278,MR1651416,MR3047087} to monotone vector fields and also improves that of \cite{MR1979730} when reduced to an Hadamard manifold, taken into account the equivalent formulation studied by Wang et al. \cite[Corollary 3.8]{MR2720608}.

The rest of this paper is organized as follows. In Section 2, we collect background materials which are required for our main results in subsequent sections. Especially, the definitions and properties of a $\CAT(0)$ space and its tangent spaces are given here. Moreover, several useful inequalities and estimates are also derived thereof. In Section 3, we give the definition of a monotone vector field and provide fundamental observations accordingly. Moreover, we investigate a convex subdifferential as an example for a monotone vector field. In Section 4, the notion of resolvents for a monotone vector field is given and several tools are obtained. We then further construct the Yosida approximations, which will be used as a central equipment in the generation theorem. Important results for this section include the resolvent identity, the convexity of closed domains for monotone vector fields, asymptotic convergence on both ends for the resolvent operators, and the Yosida approximation estimate. The final Section 5, the generation theorem is proved and a simple convergence for the corresponding trajectories are derived.

\section{Preliminaries}

Recall that a metric space $(X,d)$ is said to be \emph{geodesic} if for each two points $x,y \in X$, there exists a mapping $c : [0,\ell] \to X$ (with $\ell \geq 0$) and a constant $K \geq 0$ such that $c(0) = x$, $c(\ell) = y$, and $d(c(t),c(t')) = K\abs{t-t'}$ for any $t,t' \in [0,\ell]$. In this way, $c$ is called a \emph{geodesic} joining $x$ and $y$, and it is said to \emph{issue from} $x$ and \emph{end} at $y$. We say that $c$ is a \emph{zero geodesic} at $x$ if $x = y$, otherwise we say that it is \emph{nonzero}. Let us adopt the zero geodesic indicator $\zeta(c) := 0$ if $c$ is a zero geodesic and $\zeta(c) := 1$ otherwise. We say that $c$ is \emph{normalized} if $\ell = 1$ and that it is of \emph{unit speed} of $K = 1$. In the latter case, we also have $K = \rho(x,y)$. On the other hand, a mapping $c : \R \to X$ is called a \emph{geodesic line} if there exists a strictly positive constant $K > 0$ such that $\rho(c(t),c(t')) = K\abs{t-t'}$ for all $t,t' \in \R$.

In the sequel where the choice of geodesics is insignificant or where the geodesic is unique, we write $\gamma_{x,x}$ to denote the zero normalized geodesic at $x$, and we write $\gamma_{x,y}$ to denote the nonzero normalized geodesic joining $x$ and $y$ with $x \neq y$. We also write $\llbracket x,y \rrbracket$ to denote the image of $\gamma_{x,y}$ over $[0,1]$.

Here and henceforth, let $(\mathbb{E}^{2},\product{\cdot}{\cdot})$ be the Euclidean plane with usual inner product $\product{u}{v} := u^{\top}v$ and the Euclidean norm $\norm{u} := \sqrt{\product{u}{u}}$, for $u,v \in \mathbb{E}^{2}$. For each points $p,q,r \in X$, the geodesic triangle $\Delta \subset X$ is defined by $\Delta(pqr) := \llbracket p,q \rrbracket \cup \llbracket q,r \rrbracket \cup \llbracket r,p \rrbracket$. The triangle defined by $\overline{\Delta}(\bar{p}\bar{q}\bar{r}) := \Delta(\bar{p}\bar{q}\bar{r})$ with $\bar{p},\bar{q},\bar{r} \in \mathbb{E}^{2}$ is said to be a \emph{Euclidean comparison triangle}, or simply a \emph{comparison triangle}, if $\norm{\bar{p} - \bar{q}} = d(p,q)$, $\norm{\bar{q} - \bar{r}} = d(q,r)$, and $\norm{\bar{r} - \bar{p}} = d(r,p)$. Note that the triangle inequality of $d$ implies the existence of such comparison triangle. Moreover, the comparison triangle of each geodesic triangle in $X$ is unique up to rigid motions. Suppose that $\Delta(pqr) \subset X$ is a geodesic triangle whose comparison triangle is $\overline{\Delta}(\bar{p}\bar{q}\bar{r})$. Given $u \in \llbracket p,q \rrbracket$, the point $\bar{u} \in \llbracket \bar{p},\bar{q} \rrbracket$ is said to be a \emph{comparison point} of $u$ if $\norm{\bar{p} - \bar{u}} = d(p,u)$. Comparison points for $u' \in \llbracket q,r \rrbracket$ and $u'' \in \llbracket r,p \rrbracket$ are defined likewise.

\begin{dfn}
A geodesic metric space $(X,d)$ is said to be a \emph{$\CAT(0)$ space} if for each geodesic triangle $\Delta \subset X$ and two points $u,v \in \Delta$, the following $\CAT(0)$ inequality holds:
\[
d(u,v) \leq \norm{\bar{u} - \bar{v}},
\]
where $\bar{u},\bar{v} \in \overline{\Delta}$ are the comparison points of $u$ and $v$, respectivly, and $\overline{\Delta} \subset \mathbb{E}^{2}$ is a comparison triangle of $\Delta$. A complete $\CAT(0)$ space is also called an \emph{Hadamard space}.
\end{dfn}

The following proposition gives useful characterizations of the $\CAT(0)$ inequality. 
\begin{prop}\label{Prop:CAT(0)Char}
Suppose that $(X,d)$ is a geodesic metric space. Then, the following conditions are equivalent:
\begin{enumerate}[label=\rm{(\roman*)}]
\item $X$ is a $\CAT(0)$ space.
\item\label{cdn:CN} For all $v \in X$ and a normalized geodesic $\gamma : [0,1] \to X$, the following \eqref{eqn:CN} inequality holds for any $t \in [0,1]$:
\[\label{eqn:CN}
d^{2}(\gamma(t),v) \leq (1-t)d^{2}(\gamma(0),v) + td^{2}(\gamma(1),v) - t(1-t)d^{2}(\gamma(0),\gamma(1)). \tag{CN}
\]
\item\label{cdn:quad-characterization-CAT(0)} For all $x,y,u,v \in X$, the following inequality holds:
\begin{equation}\label{eqn:quad-characterization-CAT(0)}
d^{2}(x,v) + d^{2}(y,u) \leq d^{2}(x,u) + d^{2}(y,v) + 2d(x,y)d(u,v).
\end{equation}
\end{enumerate}
\end{prop}

Unless otherwise specified, always assume that $(H,\rho)$ is a complete $\CAT(0)$ space. Note that $H$ is always uniquely geodesic and a subset $C \subset H$ is called \emph{convex} if $\llbracket x,y \rrbracket \subset C$ for all $x,y \in C$.

\begin{thm}[\cite{MR3241330}]\label{thm:projection}
Let $C \subset H$ be nonempty, closed, and convex, and $P_{C} : H \multimap C$ be a mapping defined by
\[
P_{C}(x) := \argmin_{y \in C} \rho(x,y) \quad (\forall x \in H).
\]
Then, the following assertions hold:
\begin{enumerate}[label=\rm{(\roman*)}]
\item $P_{C}$ is defined for all $x \in H$ and is single-valued.
\item\label{cdn:ProjIneq} If $x \in H$ and $z \in C$, then
\[
\rho^{2}(x,P_{C}(x)) + \rho^{2}(P_{C}(x),z) \leq \rho^{2}(x,z).
\]
\item\label{cdn:ProjNX} $P_{C} : H \to C$ is nonexpansive.
\end{enumerate}
\end{thm}

\subsection{$\Delta$-Convergence}

If $(x^{k})$ is a bounded sequence in $H$, then the functional $H \ni x \mapsto \limsup_{k} \rho(x,x^{k})$ is finite and admits a unique minimizer \cite{MR2232680}. Such unique minimizer is called the \emph{asymptotic center} of $(x^{k})$. In this case, it is immediate that if $z = \argmin_{x \in H} \limsup_{k} \rho(x,x^{k})$, then the following \emph{Opial property} holds: $\limsup_{k} \rho(z,x^{k}) < \limsup_{k} \rho(u,x^{k})$ for all $u \in H \setminus\{z\}$. This asymptotic center is used in defining the so-called $\Delta$-convergence in the following definition.
\begin{dfn}[\cite{MR2416076}]
A bounded sequence $(x^{k})$ in $H$ is said to be \emph{$\Delta$-convergent} to a point $x \in H$ if $x$ is the (unique) asymptotic center of all subsequences of $(x^{k})$. In this case, $x$ is said to be the \emph{$\Delta$-sequential limit} of $(x^{k})$ \cite[Corollary 3.2.4]{MR3241330}.
\end{dfn}

We say a function $F : H \to \R$ is called \emph{$\Delta$-lower semicontinuous} at $\bar{x} \in H$ if
\[
F(\bar{x}) \geq \liminf_{k \tendsto \infty} F(x^{k})
\]
for any sequence $(x^{k})$ in $H$ with $\Delta$-sequential limit $\bar{x}$. In particular, the function $\rho(\cdot,x)$ is $\Delta$-lower semicountinuous on $H$ for each fixed $x \in H$.

Given any bounded sequence $(x^{k})$ in $H$, we write $\omega_{\Delta}(x^{k})$ to denote the set of all \emph{$\Delta$-subsequential limits} of $(x^{k})$, i.e., the set of all possible $\Delta$-sequential limits of subsequences of $(x^{k})$. In every $\CAT(0)$ spaces, such a set is nonempty for all bounded sequence \cite{MR3047087}. It is evident that strong convergence (i.e., convergence in metric) implies $\Delta$-convergence, but the converse implication is not true in general.
\begin{prop}\label{prop:singletonOmegaDelta}
Let $z \in H$ and $(x^{k})$ be a bounded sequence in $H$ with $\omega_{\Delta}(x^{k}) = \{z\}$. Then, $(x^{k})$ is $\Delta$-convergent to $z$.
\end{prop}
\begin{proof}
Let us suppose to the contrary that $(x^{k})$ is not $\Delta$-convergent to $z$. Hence, there exists a subsequence $(u^{k})$ of $(x^{k})$ such that the asymptotic center $u \in H$ of $(u^{k})$ is different from the point $z$.

Let $(v^{k})$ be a subsequence of $(u^{k})$ such that $\limsup_{k} \rho(z,v^{k}) = \limsup_{k} \rho(z,u^{k})$ and that $(v^{k})$ is $\Delta$-convergent to some point $v \in H$. Since $(x^{k})$ has only one $\Delta$-subsequential limit, it must be the case that $v = z$. Using the Opial property, the fact that $(v^{k}) \subset (u^{k})$, and the definition of asymptotic center, we obtain
\begin{align*}
\limsup_{k \tendsto \infty} \rho(z,u^{k}) &=  \limsup_{k \tendsto \infty} \rho(z,v^{k}) 
< \limsup_{k \tendsto \infty} \rho(u,v^{k}) \\
&\leq \limsup_{k \tendsto \infty} \rho(u,u^{k}) 
< \limsup_{k \tendsto \infty} \rho(z,u^{k}),
\end{align*}
which is a contradiction. Therefore, $(x^{k})$ must be $\Delta$-convergent to $z$.
\end{proof}

\begin{lem}[Kadec-Klee property \cite{MR2754196}]\label{lem:Kadec-Klee}
Suppose that $(x^{k})$ is $\Delta$-convergent to $x \in H$ and $\lim_{k} \rho(x^{k},p) = \rho(x,p)$ for some $p \in H$, then $(x^{k})$ is strongly convergent to $x$.
\end{lem}

Combining the $\Delta$-convergence with the strong convergence implies the following Demiclosedness Principle, which is an important result in fixed point theory.
\begin{thm}[\cite{MR2416076}]\label{thm:demiclosed}
Suppose that $C \subset H$ is nonempty, closed, and convex, and $T : C \to C$ is nonexpansive. If $(x^{k})$ is a sequence in $H$ that is $\Delta$-convergent to $x$ and that $\lim_{k} \rho(x^{k},Tx^{k}) = 0$, then $Tx = x$.
\end{thm}

The following concept of convergence and its properties are very essential in our works.
\begin{dfn}
A sequence $(x^{k})$ (resp. net $(x_{t})_{t \geq 0}$) in $H$ is said to be \emph{Fej\'er monotone} with respect to a nonempty set $V \subset H$ if for each $x \in V$, we have $\rho(x^{k+1},x) \leq \rho(x^{k},x)$ for all $k \in \N$ (resp. $\rho(x_{t},x) \leq \rho(x_{s},x)$ for all $0\leq s \leq t$).
\end{dfn}

\begin{prop}[\cite{MR3241330}]\label{prop:Fejer}
Suppose that $(x^{k}) \subset H$ is Fej\'er monotone with respect to a nonempty set $V \subset H$. Then, the following are true:
\begin{enumerate}[label=\rm{(\roman*)}]
\item $(x^{k})$ is bounded.
\item $(\rho(x,x^{k}))$ converges for any $x \in V$.
\item If every $\Delta$-accumulation point lies within $V$, then $(x^{k})$ is $\Delta$-convergent to an element in $V$.
\end{enumerate}
\end{prop}
\begin{rmk}
We can replace the sequence $(x^{k})$ in the above proposition also with a net $(x_{t})_{t \geq 0}$ and still obtain similar results.
\end{rmk}

\subsection{Tangent Spaces}

Tangent spaces (also called tangent cones) to a given $\CAT(0)$ space were introduced earlier in \cite{MR1317739} (see also \cite{MR1744486,MR1835418}). However, we make a slight modification on their representations in this paper for the future technical convenience in our studies.

In order to introduce the tangent space and related notions subsequently, we first recall the notion of comparison angle with respect to $\mathbb{E}^{2}$.
\begin{dfn}
Suppose that $p,q,r \in H$. The \emph{comparison angle} between $q$ and $r$ at $p$, denoted with $\overline{\angle}_{p}(q,r)$, is given as follows: If $q,r \in H \setminus \{p\}$, we set
\[
\cos \overline{\angle}_{p}(q,r) := \frac{\product{\bar{q}-\bar{p}}{\bar{r}-\bar{p}}}{\norm{\bar{q}-\bar{p}}\norm{\bar{r}-\bar{p}}},
\]
where $\overline{\Delta}(\bar{p},\bar{q},\bar{r}) \subset \mathbb{E}^{2}$ is the comparison triangle of the geodesic triangle $\Delta(p,q,r) \subset H$. On the other hand, we set $\overline{\angle}_{p}(p,p) := 0$, and  $\overline{\angle}_{p}(p,r) = \overline{\angle}_{p}(r,p) := \frac{\pi}{2}$ for $r \in X\setminus\{p\}$.
\end{dfn}

Given two geodesics $\gamma_{1},\gamma_{2}$ on $H$ issuing from a common point $p \in H$. The \emph{Alexandrov angle} between the two geodesics is then defined by
\[
\alpha_{p} (\gamma_{1},\gamma_{2}) := \lim_{s,t \tendsto 0^{+}} \overline{\angle}_{p} (\gamma_{1}(s),\gamma_{2}(t)).
\]
To effectively compute the Alexandrov angle, the \emph{First Variation Formula} is available in the following form:
\begin{lem}[First Variation Formula]\label{thm:FVF}
Suppose that $p \in H$, $u \in H \setminus \{p\}$, and $\gamma$ is a nonzero unit-speed geodesic issuing from $p$. Then the following identity holds:
\[
\lim_{s \tendsto 0^{+}} \frac{\rho(u,p) - \rho(u,\gamma(s))}{s} = \cos \alpha_{p}(\gamma_{p,u},\gamma).
\]
\end{lem}

Recall that the \emph{metric identification} of a pseudometric space $(\tilde{M},\tilde{d})$ is a metric space $(M,d)$, where $M$ consists of equivalence classes $[x] := \{y \in X \,|\, \tilde{d}(x,y) = 0\}$ of $x \in X$ and $d([x],[y]) := \tilde{d}(x,y)$ for all $[x],[y] \in M$.

Denoted by $\tilde{S}_{p}$ the set of all normalized geodesics issuing from $p \in H$. Then $\tilde\angle_{p} := \alpha_{p}$ defines a pseudometric on $\tilde{S}_{p}$. The metric identification of $(\tilde{S}_{p},\tilde\angle_{p})$, denoted by $(S_{p},\angle_{p})$, is called the \emph{space of directions at $p$}. In the sequel, we write $\gamma \equiv [\gamma]$ for elements of $S_{p}$. Suppose that $\sim$ is an equivalence relation on $[0,\infty) \times S_{p}$ such that $(t_{1},\gamma_{1}) \sim (t_{2},\gamma_{2})$ if and only if one of the following conditions is satisfied:
\begin{enumerate}[label=(T\arabic*)]
\item\label{cdn:T1} $t_{1}\zeta(\gamma_{1}) = t_{2}\zeta(\gamma_{2}) = 0$ or
\item\label{cdn:T2} $t_{1}\zeta(\gamma_{1}) = t_{2}\zeta(\gamma_{2}) > 0$ and $\gamma_{1} = \gamma_{2}$.
\end{enumerate}
Put $T_{p}H := \left([0,\infty) \times S_{p}\right) / \sim$ and whenever there is no ambiguity, let us write $t\gamma \equiv [(t,\gamma)]_{\sim}$ for elements of $T_{p}H$ to simplify the notions.

Next, we endow $T_{p}H$ with a metric $d_{p}$ defined for each $t_{1}\gamma_{1},t_{2}\gamma_{2} \in T_{p}H$ by
\[
d_{p}(t_{1}\gamma_{1},t_{2}\gamma_{2}) := \sqrt{t_{1}^{2}\zeta(\gamma_{1}) + t_{2}^{2}\zeta(\gamma_{2}) - 2t_{1}t_{2}\zeta(\gamma_{1})\zeta(\gamma_{2})\cos \angle_{p}(\gamma_{1},\gamma_{2})}.
\]
To see the metric properties of $d_{p}$, we first note that the inequalities
\begin{equation}\label{eqn:aux_trig}
\abs{t_{1}\zeta(\gamma_{1}) - t_{2}\zeta\gamma_{2})} \leq d_{p}(t_{1}\gamma_{1},t_{2}\gamma_{2}) \leq t_{1}\zeta(\gamma_{1}) + t_{2}\zeta(\gamma_{2})
\end{equation}
hold for any $t_{1}\gamma_{1},t_{2}\gamma_{2} \in T_{p}H$. Let $\mathbb{T}_{p}H := \left([0,\infty) \times S'_{p}\right)/\approx$, where $S'_{p} := \{\gamma \in S_{p} \,|\, \zeta(\gamma) = 1\}$ and $\approx \,:=\, \sim \upharpoonright _{\mathbb{T}_{p}H}$ is the restricted equivalence relation. One can see easily that $[(t,\gamma)]_{\approx} = [(t,\gamma)]_{\sim}$ whenever $t > 0$ and $\gamma \in S'_{p}$. Let us put
\[
\left\{\begin{array}{l}
X_{+} := \{[(t,\gamma)]_{\sim} \,|\, t > 0, \zeta(\gamma) = 1\}; \\[.5em]
X_{0} := \{[(0,\gamma)]_{\sim}\} = \big\{\{(t,\gamma) \,|\, t = 0 \vee \zeta(\gamma) = 0\}\big\}; \\[.5em]
X_{0}' := \{[(0,\gamma)]_{\approx}\} = \big\{\{(t,\gamma) \,|\, t = 0, \zeta(\gamma) = 1\}\big\}.
\end{array}
\right.
\]
We can then write $T_{p}H$ and $\mathbb{T}_{p}H$ with the following representations:
\begin{equation}\label{eqn:representation}
\text{$T_{p}H = X_{+} \cup X_{0}$ \,\, and \,\, $\mathbb{T}_{p}H = X_{+} \cup X_{0}'.$}
\end{equation}

According to \cite{MR1744486} and \cite{MR1835418}, $\mathbb{T}_{p}H$ is a metric space with respect to a metric $D_{p}$ given by
\[
D_{p}([(t_{1},\gamma_{1})]_{\approx},[(t_{2},\gamma_{2})]_{\approx}) := \sqrt{t_{1}^{2} + t_{2}^{2} - 2t_{1}t_{2}\cos \angle_{p}(\gamma_{1},\gamma_{2})}, \quad \forall [(t_{1},\gamma_{1})]_{\approx},[(t_{2},\gamma_{2})]_{\approx} \in \mathbb{T}_{p}H.
\]
We will show now that $(T_{p}H,d_{p})$ is a metric space and it is isometry to $(\mathbb{T}_{p}H,D_{p})$.

\begin{prop}
$d_{p}$ is a metric on $T_{p}H$.
\end{prop}
\begin{proof}
Let $t_{1}\gamma_{1},t_{2}\gamma_{2} \in T_{p}H$. By \eqref{eqn:aux_trig}, we can see that
\begin{align*}
d_{p}(t_{1}\gamma_{1},t_{2}\gamma_{2}) = 0 &\iff \text{$t_{1}\zeta(\gamma_{1}) = t_{2}\zeta(\gamma_{2})$ and $t_{1}^{2}(1 - \cos \angle_{p} (\gamma_{1},\gamma_{2})) = 0$} \\
&\iff \text{either \ref{cdn:T1} or \ref{cdn:T2} holds} \\
&\iff t_{1}\gamma_{1} = t_{2}\gamma_{2}.
\end{align*}
Next, since $\angle_{p}$ is a metric on $S_{p}$, we obtain $d_{p}(t_{1}\gamma_{1},t_{2}\gamma_{2}) = d_{p}(t_{2}\gamma_{2},t_{1}\gamma_{1})$ immediately. It remains to show that the triangle inequality
\[
d_{p}(t_{1}\gamma_{1},t_{3}\gamma_{3}) \leq d_{p}(t_{1}\gamma_{1},t_{2}\gamma_{2}) + d_{p}(t_{2}\gamma_{2},t_{3}\gamma_{3})
\]
holds for any $t_{i}\gamma_{i} \in T_{p}H$, with $i = 1,2,3$. If $t_{i} > 0$ and $\zeta(\gamma_{i}) = 1$ for all $i = 1,2,3$, we can use the fact that $d_{p}$ is a metric on $X_{+}$ to obtain the triangle inequality. If $t_{1} = 0$ or $\zeta(\gamma_{1}) = 0$, then it follows from \eqref{eqn:aux_trig} that
\begin{align*}
d_{p} (t_{1}\gamma_{1},t_{3}\gamma_{3}) &= t_{3}\zeta(\gamma_{3}) \\
&\leq t_{2}\zeta(\gamma_{2}) + \abs{t_{2}\zeta(\gamma_{2}) - t_{3}\zeta(\gamma_{3}) } \\
&\leq t_{2}\zeta(\gamma_{2}) + d_{p}(t_{2}\gamma_{2},t_{3}\gamma_{3}) \\
&= d_{p}(t_{1}\gamma_{1},t_{2}\gamma_{2}) + d_{p}(t_{2}\gamma_{2},t_{3}\gamma_{3}).
\end{align*}
Similar procedure also works when $t_{3} = 0$ or $\zeta(\gamma_{3}) = 0$. Finally, if $t_{2} = 0$ or $\zeta(\gamma_{2}) = 0$, then \eqref{eqn:aux_trig} implies
\[
d_{p}(t_{1}\gamma_{1},t_{3}\gamma_{3}) \leq t_{1}\zeta(\gamma_{1}) + t_{3}\zeta(\gamma_{3}) = d_{p}(t_{1}\gamma_{1},t_{2}\gamma_{2}) + d_{p}(t_{2}\gamma_{2},t_{3}\gamma_{3}).
\]
Hence, we now conclude that $d_{p}$ is a metric on $T_{p}H$.
\end{proof}

\begin{prop}
There is a bijection which preserves distances between $T_{p}H$ and $\mathbb{T}_{p}H$.
\end{prop}
\begin{proof}
We may see from \eqref{eqn:representation} that a mapping $\varphi_{p} : \mathbb{T}_{p}H \to T_{p}H$ given by
\[
\varphi_{p} \left([(t,\gamma)]_{\approx}\right) := \left\{\begin{array}{ll}
[(t,\gamma)]_{\sim}	 	&\text{if $[(t,\gamma)]_{\approx} \in X_{+}$;} \\[.5em]
[(0,\gamma)]_{\sim}		&\text{if $[(t,\gamma)]_{\approx} \in X_{0}'$}
\end{array}
\right.
\]
is a bijection. That is, $\varphi_{p}$ is identity on $X_{+}$ and maps $X_{0}'$ onto $X_{0}$. The fact that $\varphi_{p}$ preserves distances is trivial. 
\end{proof}

The metric space $(T_{p}H,d_{p})$ is henceforth called the \emph{tangent space} of $H$ at $p$. The \emph{tangent bundle} of $H$ is then defined by $TH := \bigcup_{p \in H} T_{p}H$. The isometry result above ensures that it is consistent with the classical notion of tangent spaces of a complete $\CAT(0)$ space as was given by \cite{MR1317739}. Moreover, we also have further implication in cases of $H$ being a Hilbert space or an Hadamard manifold. For instance, if $\mathcal{H}$ is a Hilbert space and $p \in \mathcal{H}$, then $T_{p}\mathcal{H}$ is isometric with $\mathcal{H}-\{p\}$ (and hence to $\mathcal{H}$) by the canonical map $t\gamma \mapsto \gamma(t) - p$. On the other hand, let $M$ be an Hadamard manifold with Riemannian tangent space at $p \in M$ denoted by $\mathds{T}_{p}M$. We know, in this case, that the exponential map $\exp_{p} : \mathds{T}_{p}M \to M$ is well-defined on the whole tangent space and is a diffeomorphism. Then $T_{p}M$ is isometric with $\mathds{T}_{p}M$ by the canonical map $t\gamma \mapsto \exp_{p}^{-1} t\dot\gamma(0)$, where we use the convention $\dot\gamma_{p,p}(0) = 0_{\mathds{T}_{p}M}$. See also \cite{MR1744486,MR1835418} and references therein for further information.

On each tangent space $T_{p}H$ we write $0_{p} := 0\gamma = t\gamma_{p,p}$ (here, $\gamma \in S_{p}$ and $t > 0$) and $\norm{t\gamma}_{p} := d_{p}(0_{p},t\gamma) = t\zeta(\gamma)$. For convenience, we invoke the notion $\0 := \{0_{p} \,|\, p \in H\}$ the \emph{zero section} of $TH$. Moreover, we adopt the product
\[
g_{p}(t_{1}\gamma_{1},t_{2}\gamma_{2}) := \frac{1}{2} \left[ \norm{t_{1}\gamma_{1}}_{p}^{2} + \norm{t_{2}\gamma_{2}}_{p}^{2} - d_{p}^{2}(t_{1}\gamma_{1},t_{2}\gamma_{2})\right]
\]
for any $t_{1}\gamma_{1},t_{2}\gamma_{2} \in T_{p}H$. By a direct calculation, we can deduce that
\[
g_{p}(t_{1}\gamma_{1},t_{2}\gamma_{2}) = t_{1}t_{2}\zeta(\gamma_{1})\zeta(\gamma_{2})\cos \angle_{p}(\gamma_{1},\gamma_{2}) \leq \norm{t_{1}\gamma_{1}}_{p}\norm{t_{2}\gamma_{2}}_{p},
\]
 which is an analogue of the Cauchy-Schwarz inequality. It is easy to see that 
\[
g_{p}(t_{1}\gamma_{1},t_{1}\gamma_{1}) = \norm{t_{1}\gamma_{1}}_{p}^{2}, \quad g_{p}(t_{1}\gamma_{1},t_{2}\gamma_{2}) = g_{p}(t_{2}\gamma_{2},t_{1}\gamma_{1}), \quad \text{and \,\,} g_{p}(t_{1}\gamma_{1},t_{2}\gamma_{2}) = t_{1}g_{p}(\gamma_{1},t_{2}\gamma_{2}).
\]

The following inequality is of fundamental importance in this present paper. Here, we adopt the notation
\[
\product{t\vec{px}}{s\vec{py}} := \frac{ts}{2} \left[\rho^{2}(p,x) + \rho^{2}(p,y) - \rho^{2}(x,y)\right]
\]
for $p,x,y \in H$ and $t,s \geq 0$.

\begin{prop}\label{Prop:gp-ql}
For each $s,t \geq 0$ and $p,x,y \in H$, the following inequality holds:
\[
g_{p}\left(t\rho(p,x)\gamma_{p,x},s\rho(p,y)\gamma_{p,y}\right) \geq \product{t\vec{px}}{s\vec{py}}.
\]
\end{prop}
\begin{proof}
The case where one of $\gamma_{p,x}$ and $\gamma_{p,y}$ is a zero geodesic is obvious. Suppose that both of them are nonzero geodesics. Since $\angle_{p}(\gamma_{p,x},\gamma_{p,y}) \leq \overline{\angle}_{p}(\gamma_{p,x},\gamma_{p,y})$, we have
\begin{align*}
\lefteqn{g_{p}\left(t\rho(p,x)\gamma_{p,x},s\rho(p,y)\gamma_{p,y}\right)} \\
&= st\rho(p,x)\rho(p,y) \cos \angle_{p}(\gamma_{p,x},\gamma_{p,y}) \\
&\geq st\rho(p,x)\rho(p,y) \cos \overline\angle_{p}(x,y) \\
&= st\rho(p,x)\rho(p,y) \left[ \frac{1}{2\rho(p,x)\rho(p,y)} \left( \rho^{2}(p,x) + \rho^{2}(p,y) - \rho^{2}(x,y) \right) \right] \\
&= \product{t\vec{px}}{s\vec{py}}. \tag*\qedhere
\end{align*}
\end{proof}

\section{Monotone vector fields}

In this section, we give a systematic study of the class of monotone vector fields on a $\CAT(0)$ space $H$ together with the two main properties, the maximality and the surjectivity condition. To emphasize the practical use of this class of vector fields, we also dedicate an especial observation of the subdifferential for a proper convex lower semicontinuous function as a monotone vector field.

By a \emph{(set-valued) vector field on $H$}, we mean a mapping $A : H \multimap TH$ satisfying $Ap \subset T_{p}H$ for all $p \in H$.

\begin{dfn}
A vector field $A : H \multimap TH$ is said to be \emph{monotone} if
\[
g_{p}(\eta,\gamma_{p,q}) \leq -g_{q}(\nu,\gamma_{q,p})
\]
holds for every $(p,\eta),(q,\nu) \in \gr(A)$, where $\gr(A) := \{(x,u) \in H \times TH \,|\, u \in Ax\}$ denotes the \emph{graph} of $A$. In addition, if $\gr(A)$ is not properly contained in the graph of any other monotone vector fields, then $A$ is said to be \emph{maximally monotone}.
\end{dfn}

\begin{prop}
Let $A$ be a monotone vector field on $H$ and $(p,\eta),(q,\nu) \in \gr(A)$. Suppose that the representations $t\rho(p,x)\gamma_{p,x} = \eta$ and $s\rho(q,y)\gamma_{q,y} = \nu$ are satisfied for some $s,t \geq 0$, and $x,y \in H$. Then the following inequality holds:
\[
\product{t\vec{px}}{\vec{pq}} + \product{s\vec{qy}}{\vec{qp}} \leq 0
\]
\end{prop}
\begin{proof}
The inequality is obvious if $p = q$. Thus we assume that $p \neq q$. It yields immediately from the monotonicity of $A$ and Proposition \ref{Prop:gp-ql} that
\begin{align*}
0 &\geq g_{p}(\eta,\gamma_{p,q}) + g_{q}(\nu,\gamma_{q,p}) \\
&= g_{p}(t\rho(p,x)\gamma_{p,x},\gamma_{p,q}) + g_{q}(s\rho(q,y)\gamma_{q,y},\gamma_{q,p}) \\
&\geq \frac{1}{\rho(p,q)}\left[ \product{t\vec{px}}{\vec{pq}} + \product{s\vec{qy}}{\vec{qp}} \right].
\end{align*}
The desired inequality then follows by rearrangements.
\end{proof}

The following definition is central in the studies in the rest of this paper.
\begin{dfn}
A vector field $A : H \multimap TH$ is said to satisfy the \emph{surjectivity condition} if for any $t > 0$ and $x \in H$, there exists a point $p \in H$ such that $t\rho(p,x)\gamma_{p,x} \in Ap$.
\end{dfn}

The monotonicity and surjectivity conditions are the two main ingredients for our theory developed henceforth in this paper.

\begin{prop}\label{prop:uniqep}
If $A$ is a monotone vector field on $H$ with surjectivity condition, then for any $t > 0$ and $x \in H$, there exists a unique $p \in H$ such that $t\rho(p,x)\gamma_{p,x} \in Ap$.
\end{prop}
\begin{proof}
Let $t > 0$, $x \in H$, and $p,p' \in H$ be points where $t\rho(p,x)\gamma_{p,x} \in Ap$ and $t\rho(p',x)\gamma_{p',x} \in Ap'$. By the monotonicity of $A$ and Proposition \ref{Prop:gp-ql}, we have
\begin{align*}
0 &\geq \rho(p,p'
)\left[g_{p}(t\rho(p,x)\gamma_{p,x},\gamma_{p,p'}) + g_{p'}(t\rho(p',x)\gamma_{p',x},\gamma_{p',p})\right] \\
&= g_{p}(t\rho(p,x)\gamma_{p,x},\rho(p,p')\gamma_{p,p'}) + g_{p'}(t\rho(p',x)\gamma_{p',x},\rho(p,p')\gamma_{p',p}) \\
&\geq t \left[ \product{\vec{px\vphantom{p'}}}{\vec{pp'}} + \product{\vec{p'x}}{\vec{p'p}} \right] \\
&= t \left[ \product{\vec{px\vphantom{p'}}}{\vec{pp'}} + \product{\vec{xp'}}{\vec{pp'}} \right] \\
&= t\product{\vec{pp'}}{\vec{pp'}} \\
&= t\rho^{2}(p,p').
\end{align*}
Hence $p = p'$ and the statement is proved.
\end{proof}

\begin{prop}\label{prop:surj-mm}
If $A$ is a monotone vector field over $H$ with surjectivity condition, then it is maximally monotone.
\end{prop}
\begin{proof}
Suppose that $p \in H$ and $p^{\ast} \in T_{p}H$ satisfy the inequality
\begin{equation}\label{eqn:premax}
g_{p}(p^{\ast},\gamma_{p,q}) + g_{q}(q^{\ast},\gamma_{q,p}) \leq 0,
\end{equation}
for any choice of $(q,q^{\ast}) \in \gr(A)$. Suppose that $p^{\ast}$ has the representation $p^{\ast} = t\rho(p,u)\gamma_{p,u}$ for some $t > 0$ and $u \in H$. By the surjectivity condition of $A$, there is a unique point $z \in H$ such that $t\rho(z,u)\gamma_{z,u} \in Az$. In view of \eqref{eqn:premax} and Proposition \ref{Prop:gp-ql}, we have
\begin{align*}
0 &\geq \rho(p,z)\left[ g_{p}(t\rho(p,u)\gamma_{p,u},\gamma_{p,z}) + g_{z}(t\rho(z,u)\gamma_{z,u},\gamma_{z,p}) \right] \\
&= g_{p}(t\rho(p,u)\gamma_{p,u},\rho(p,z)\gamma_{p,z}) + g_{z}(t\rho(z,u)\gamma_{z,u},\rho(p,z)\gamma_{z,p}) \\
&\geq t \left[ \product{\vec{pu}}{\vec{pz}} + \product{\vec{zu}}{\vec{zp}} \right] \\
&= t \left[ \product{\vec{pu}}{\vec{pz}} + \product{\vec{uz}}{\vec{pz}} \right] \\
&= t \product{\vec{pz}}{\vec{pz}} \\
&= t\rho^{2}(p,z).
\end{align*}
Since $t > 0$, the above inequalities imply $p = z$ so that $(p,p^{\ast}) \in \gr(A)$. Hence the maximality is obtained.
\end{proof}

\subsection{Subdifferential of a convex function}

In this subsection, we study a particular example of a monotone vector field with surjectivity condition, namely the subdifferential of a proper, convex, lower semicontinuous function. For simplicity, we write $\Gamma_{0}$ to denote the class of proper, convex, lower semicontinuous functions $F : H \to (-\infty,+\infty]$.

\begin{dfn}
Let $F \in \Gamma_{0}$ be given. At each $p \in H$, a tangent vector $p^{\ast} \in T_{p}H$ is called a \emph{subgradient} of $F$ at $p$ if 
\[
F(x) \geq F(p) + \rho(p,x)g_{p}(p^{\ast},\gamma_{p,x})
\]
for every $x \in H$. The \emph{subdifferential} of $F$ is the vector field $\partial F : H \multimap TH$, where $\partial F (p)$ is the set of all subgradients of $F$ at $p$ for each $p \in H$.
\end{dfn}

\begin{prop}
$\partial F$ is monotone for each $F \in \Gamma_{0}$.
\end{prop}
\begin{proof}
Suppose that $F \in \Gamma_{0}$ and $(p,p^{\ast}),(q,q^{\ast}) \in \gr(\partial F)$. We thus have
\[
F(q) \geq F(p) + \rho(p,q)g_{p}(p^{\ast},\gamma_{p,q})
\]
and also
\[
F(p) \geq F(q) + \rho(p,q)g_{q}(q^{\ast},\gamma_{q,p}).
\]
Rearraging yields
\[
-\rho(p,q)g_{q}(q^{\ast},\gamma_{q,p}) \geq F(q) - F(p) \geq \rho(p,q)g_{p}(p^{\ast},\gamma_{p,q}),
\]
and so the monotonicity is obtained.
\end{proof}

The following result characterizes elements in $\partial F$ and can also be regarded as a generalized Fermat rule.
\begin{prop}\label{prop:genFermat}
Let $F \in \Gamma_{0}$, $p,\bar{p} \in H$ and $\lambda > 0$ be given. Then
\[
\lambda^{-1}\rho(\bar{p},p)\gamma_{\bar{p},p} \in \partial F(\bar{p}) \iff \bar{p} = \prox_{\lambda} (p) = \argmin_{y \in H} \left[ F(y) + \frac{1}{2\lambda}\rho^{2}(y,p) \right].
\]
\end{prop}
\begin{proof}
Let us first show the `only if' part. Assume that $\lambda^{-1}\rho(\bar{p},p)\gamma_{\bar{p},p} \in \partial F(\bar{p})$. For any $y \in H$, we get
\begin{align*}
\lambda F(y) &\geq \lambda F(\bar{p}) + \rho(\bar{p},p)\rho(\bar{p},y)g_{\bar{p}}(\gamma_{\bar{p},p},\gamma_{\bar{p},y}) \\
&\geq \lambda F(\bar{p}) + \product{\vec{\bar{p}p}}{\vec{\bar{p}y}} \\
&= \lambda F(\bar{p}) + \frac{1}{2} \left[ \rho^{2}(\bar{p},p) + \rho^{2}(\bar{p},y) - \rho^{2}(p,y) \right] \\
&\geq \lambda F(\bar{p}) + \frac{1}{2} \left[ \rho^{2}(\bar{p},p) - \rho^{2}(p,y) \right].
\end{align*}
Therefore, we have $\bar{p} = \prox_{\lambda} (p)$.

Next, we show the `if' part. Suppose that $\bar{p} = \prox_{\lambda}(p)$, whose definition gives
\[
F(y) - F(\bar{p}) \geq \frac{1}{2\lambda} \left[ \rho^{2}(\bar{p},p) - \rho^{2}(y,p) \right],
\]
for all $y \in H$. If $p = \bar{p}$, then $\bar{p}$ is a minimizer of $F$ (see \cite[Proposition 6.5]{MR3206460}). It follows from the definition of $\partial F$ that $\lambda^{-1}\rho(\bar{p},p)\gamma_{\bar{p},p} = 0_{\bar{p}} \in \partial F(\bar{p})$. Next, suppose that $p \neq \bar{p}$. Let us fix any $y \in H\setminus\{\bar{p}\}$ and for each $t \in [0,1]$, put $y_{t} := (1-t)\bar{p}\oplus ty$. Further, if we set $Y(s) := y_{s/\rho(\bar{p},y)}$ for each $s \in [0,\rho(\bar{p},y)]$, then $Y$ is a unit-speed geodesic.

For $t \in (0,1)$, putting $y = y_{t}$ in the above inequality and applying the convexity of $F$ yield
\begin{align*}
F(y) - F(\bar{p}) &\geq \frac{1}{2\lambda t} \left[ \rho^{2}(\bar{p},p) - \rho^{2}(y_{t},p) \right] \\
&= \frac{1}{2\lambda}[\rho(\bar{p},p) + \rho(y_{t},p)]\left[\frac{\rho(\bar{p},p) - \rho(y_{t},p)}{t}\right] \\
&= \frac{\rho(\bar{p},y)}{2\lambda}[\rho(\bar{p},p) + \rho(y_{t},p)]\left[\frac{\rho(\bar{p},p) - \rho(Y(t\rho(\bar{p},y)),p)}{t\rho(\bar{p},y)}\right].
\end{align*}
Letting $t \tendsto 0^{+}$ and taking into account the First Variation Formula (Theorem \ref{thm:FVF}), we have
\[
F(y) - F(\bar{p}) \geq \frac{\rho(\bar{p},p)\rho(\bar{p},y)}{\lambda} \cos \angle_{p}(\gamma_{\bar{p},p},\gamma_{\bar{p},y}) = \rho(\bar{p},y) g_{p}\left(\lambda^{-1}\rho(\bar{p},p)\gamma_{\bar{p},p},\gamma_{\bar{p},y}\right).
\]
It is trivial that the above inequality holds for $y = \bar{p}$. Therefore, we may conclude that $\lambda^{-1}\rho(\bar{p},p)\gamma_{\bar{p},p} \in \partial F(\bar{p})$.
\end{proof}

\begin{thm}
$\partial F$ is maximally monotone for $F \in \Gamma_{0}$.
\end{thm}
\begin{proof}
It is clear by the property of the operator $\prox$ and the previous theorem that $\partial F$ satisfies the surjectivity condition. The conclusion follows in view of Proposition \ref{prop:surj-mm}.
\end{proof}

Finally, we show the density of the domain of $\partial F$ in that of $F$. In other words, a proper, convex, lower semicontinuous function is subdifferentiable almost everywhere in reasonable measures.

\begin{thm}
$\overline{\dom \partial F} = \overline{\dom F}$.
\end{thm}
\begin{proof}
The inclusion $\overline{\dom \partial F} \subset \overline{\dom F}$ is immediately implied from the definition. Hence we only need to show the inclusion $\overline{\dom F} \subset \overline{\dom \partial F}$. Suppose that $x \in \overline{\dom F}$. According to \cite[Proposition 2.2.26]{MR3241330}, we know that $x = \lim_{\lambda \tendsto 0^{+}} \prox_{\lambda} (x)$. By Proposition \ref{prop:genFermat}, we may see that 
\[
\lambda^{-1}\rho(\prox_{\lambda}(x),x)\gamma_{\prox_{\lambda}(x),x} \in \partial F (\prox_{\lambda} (x)),
\]
for any $\lambda > 0$. This shows that $(\prox_{\lambda}(x))_{\lambda > 0}$ is a net in $\dom \partial F$. Therefore, $x$ as a limit point of this net must lies within the closure $\overline{\dom \partial F}$.
\end{proof}

\section{Resolvents and Yosida approximations}

Now, we shall define the resolvent for a given vector field $A : H \multimap TH$ and derive some of its fundamental properties. Results in this section are considered to be the main auxiliary tools used in the final section.

\begin{dfn}
Given $\lambda > 0$, the \emph{$\lambda$-resolvent} of $A$ is the mapping $J_{\lambda} : H \multimap H$ defined by
\[
J_{\lambda}x := \{z \in H \,|\, \lambda^{-1}\rho(z,x)\gamma_{z,x} \in Az\} \quad (\forall x \in H).
\]
Moreover, we define $J_{0}$ to be the identity mapping.
\end{dfn}

\begin{dfn}
A mapping $T : C \to C$ is called \emph{firmly nonexpansive} if for any $x,y \in C$, the function
\[
\varphi_{x,y}(t) := \rho(\gamma_{x,Tx}(t),\gamma_{y,Ty}(t))
\]
is nonincreasing on $t \in [0,1]$.
\end{dfn}

\begin{prop}\label{prop:properties-J}
Suppose that $A$ is a monotone vector field on $H$ satisfying the surjectivity condition. Then the following facts hold true:
\begin{enumerate}[label=\rm{(\roman*)}]
\item\label{cdn:singlevalued} $J_{\lambda}$ is well-defined on $H$ and is single-valued.
\item\label{cdn:J-NX} $J_{\lambda}$ is nonexpansive.
\item\label{cdn:Fixed=Stationary} $\Fix(J_{\lambda}) = A^{-1}(\0)$.
\item\label{cdn:J-ResolventIden} If $0 < \mu \leq \lambda$, then $J_{\lambda}(x) = J_{\mu} (u)$ with $u := \left(1 - \frac{\mu}{\lambda}\right)J_{\lambda}(x) \oplus \frac{\mu}{\lambda} x$, for each $x \in H$.
\item\label{cdn:J-FirmNX} $J_{\lambda}$ is firmly nonexpansive.
\end{enumerate}
\end{prop}
\begin{proof}
\ref{cdn:singlevalued} The well-definition as a single-valued mapping follows from the surjectivity condition and Proposition \ref{prop:uniqep}.

\ref{cdn:J-NX} Suppose that $x,y \in H$. By the monotonicity of $A$, we have
\begin{align*}
0 &\geq \rho(J_{\lambda}x,J_{\lambda}y) \left[ g_{J_{\lambda}x}(\lambda^{-1}\rho(J_{\lambda}x,x)\gamma_{J_{\lambda}x,x},\gamma_{J_{\lambda}x,J_{\lambda}y}) + g_{J_{\lambda}y}(\lambda^{-1}\rho(J_{\lambda}y,y)\gamma_{J_{\lambda}y,y},\gamma_{J_{\lambda}y,J_{\lambda}x})\right] \\
&\geq \product{\lambda^{-1}\gamma_{J_{\lambda}x,x}}{\gamma_{J_{\lambda}x,J_{\lambda}y}} + \product{\lambda^{-1}\gamma_{J_{\lambda}y,y}}{\gamma_{J_{\lambda}y,J_{\lambda}x}} \\
&= \frac{1}{2\lambda} \left( \rho^{2}(J_{\lambda}x,x) + \rho^{2}(J_{\lambda}x,J_{\lambda}y) - \rho^{2}(x,J_{\lambda}y) \right) \\
	&\qquad + \frac{1}{2\lambda}\left(\rho^{2}(J_{\lambda}y,y) + \rho^{2}(J_{\lambda}y,J_{\lambda}x) - \rho^{2}(y,J_{\lambda}x) \right).
\end{align*}
Rearranging and applying \eqref{eqn:quad-characterization-CAT(0)}, we get
\begin{align}
2\rho^{2}(J_{\lambda}x,J_{\lambda}y) &\leq \rho^{2}(x,J_{\lambda}y) + \rho^{2}(y,J_{\lambda}x) - \rho^{2}(J_{\lambda}x,x) - \rho^{2}(J_{\lambda}y,y) \label{eqn:firmlynonspreading}\\
&\leq 2\rho(x,y)\rho(J_{\lambda}x,J_{\lambda}y), \nonumber
\end{align}
which implies the nonexpansivity of $J_{\lambda}$.

\ref{cdn:Fixed=Stationary} The result is simply obtained from
\[
x \in \Fix(J_{\lambda}) \iff x = J_{\lambda}x \iff 0_{x} = \lambda^{-1}\gamma_{x,x} \in Ax \iff x \in A^{-1}(\0).	
\]

\ref{cdn:J-ResolventIden} Let $x \in H$ be arbitrary and $0 < \mu \leq \lambda$. Set $u := \left(1-\frac{\mu}{\lambda}\right)J_{\lambda}x \oplus \frac{\mu}{\lambda}x$.  If $x = J_{\lambda}x$, then $u = x \in A^{-1}\0$. It follows from \ref{cdn:Fixed=Stationary} that $J_{\lambda} x = x = u = J_{\mu} u$. Thus we suppose that $x \neq J_{\lambda}x$. In this case, we have $\angle_{J_{\lambda}x}(\gamma_{J_{\lambda}x,x},\gamma_{J_{\lambda}x,u}) = 0$ and $\zeta(\gamma_{J_{\lambda}x,x}) = \zeta(\gamma_{J_{\lambda}x,u}) = 1$. Also note that
\[
\lambda^{-1}\rho(J_{\lambda}x,x) = \mu^{-1}(\mu\lambda^{-1})\rho(J_{\lambda}x,x) = \mu^{-1}\rho(J_{\lambda}x,u).
\]
By the definition of a tangent space, we get
\[
\mu^{-1}\rho(J_{\lambda}x,u)\gamma_{J_{\lambda}x,u} = \lambda^{-1}\rho(J_{\lambda}x,x)\gamma_{J_{\lambda}x,x} \in AJ_{\lambda}x,
\]
which leads to the conclusion that $J_{\lambda}x = J_{\mu}u$.

\ref{cdn:J-FirmNX} Let $x,y \in H$. Since $t \mapsto \varphi_{x,y}(t)$ is  convex on $[0,1]$ (see \cite[Proposition 1.1.5]{MR3241330}), it is sufficient to show that $\varphi_{x,y}(1) \leq \varphi_{x,y}(t)$ for all $t \in (0,1)$. Indeed, for $t \in (0,1)$, we obtain from \ref{cdn:J-ResolventIden} and \ref{cdn:J-NX} the following:
\begin{align*}
\varphi_{x,y}(1) &= \rho(J_{\lambda}x,J_{\lambda}y) = \rho\left(J_{(1-t)\lambda}(tJ_{\lambda}x \oplus (1-t)x),J_{(1-t)\lambda}(tJ_{\lambda}y \oplus (1-t)y)\right) \\
&\leq \rho(tJ_{\lambda}x \oplus (1-t)x,tJ_{\lambda}y \oplus (1-t)y) = \varphi_{x,y}(t).
\end{align*}
Therefore, $J_{\lambda}$ is firmly nonexpansive.
\end{proof}

Henceforth in this paper, we need to assume the geodesic extension property on $H$ in order to define a \emph{negative geodesic}. Recall that a $\CAT(0)$ space $H$ has the \emph{geodesic extension property} if each unit speed geodesic $\gamma : [0,L] \to H$, there exists an isometry $\eta : \R \to H$ such that $\eta(t) = \gamma(t)$ whenever $t \in [0,L]$. It is clear that $H$ has the geodesic extension property if and only if for each $p,q \in H$, there exists $\eta : \R \to H$ such that $\rho(\eta(s),\eta(t)) = \rho(p,q)\abs{s-t}$ for all $s,t \in \R$ with $\eta(0) = p$ and $\eta(1) = q$. (See \cite{MR1744486} for a detailed description).

\begin{dfn}
Suppose that $H$ has the geodesic extension property, and that $p,x \in H$. By the geodesic extension property, there is a point $y \in H$ such that $\gamma_{x,y}(\frac{1}{2}) = p$. The \emph{negative geodesic} of $\gamma_{p,x}$ is then defined by $-\gamma_{p,x} := \gamma_{p,y}$.
\end{dfn}

\begin{rmk}
Notice that the point $y \in H$ appeared in the above definition is not necessarily unique. In this case, we fix one of such points $y$ for any given $x,z \in H$.
\end{rmk}

\begin{prop}\label{prop:negativeineq}
Suppose that $H$ has the geodesic extension property. Then
\[
g_{p}(-\gamma_{p,x},\gamma_{p,q}) \leq - g_{p}(\gamma_{p,x},\gamma_{p,q})
\]
for all $p,q,x \in H$.
\end{prop}
\begin{proof}
Suppose that $p \in H$ and $q,x \in H \setminus\{p\}$, otherwise there is nothing to be proved. Suppose that $y \in H$ satifies $\gamma_{x,y}(1/2) = p$. By the triangle inequality, it follows that
\[
\pi = \angle_{p}(\gamma_{p,x},\gamma_{p,y}) \leq \angle_{p}(\gamma_{p,x},\gamma_{p,q}) + \angle_{p}(\gamma_{p,y},\gamma_{p,q}).
\]
Rearranging the inequality yields $\angle_{p}(\gamma_{p,x},\gamma_{p,q}) \geq \pi - \angle_{p}(\gamma_{p,y},\gamma_{p,q})$, and we further have
\[
\cos \angle_{p}(\gamma_{p,x},\gamma_{p,q}) \leq \cos \left(\pi - \angle_{p}(\gamma_{p,y},\gamma_{p,q})\right) = - \cos \angle_{p}(\gamma_{p,y},\gamma_{p,q}).
\]
Using this fact, we obtain
\begin{align*}
g_{p}(-\gamma_{p,x},\gamma_{p,q}) &= g_{p}(\gamma_{p,y},\gamma_{p,q}) \\
&= \cos \angle_{p}(\gamma_{p,y},\gamma_{p,q}) \\
&\leq - \cos \angle_{p}(\gamma_{p,x},\gamma_{p,q}) \\
&= -g_{p}(\gamma_{p,x},\gamma_{p,q}). \tag*{\qedhere}
\end{align*}
\end{proof}

Now that we have defined the negative geodesic, we use it in the definition of the so-called complemenatary vector field.
\begin{dfn}
Given a mapping $T : H \to H$. The \emph{complementary vector field} of $T$, denoted by $A_{T} : H \to TH$ is defined by
\[
A_{T}x := \rho(x,Tx)(-\gamma_{x,Tx}) \quad (\forall x \in H).
\]
\end{dfn}

It is immediate to observe that $\Fix(T) = A_{T}^{-1}\0$. The next proposition shows that a complementary vector field is monotone if applied to a nonexpansive mapping.
\begin{prop}
Suppose that $H$ has the geodesic extension property. If $T : H \to H$ is nonexpansive, then $A_{T}$ is monotone.
\end{prop}
\begin{proof}
Let $x,y \in H$. In view of \eqref{eqn:quad-characterization-CAT(0)}, the nonexpansivity of $T$, Propositions \ref{Prop:gp-ql} and \ref{prop:negativeineq}, we have
\begin{align*}
\lefteqn{2\rho(x,y)[g_{x}(A_{T}x,\gamma_{x,y}) + g_{y}(A_{T}y,\gamma_{y,x})]} \\
&\leq -2\rho(x,y)[g_{x}(\rho(x,Tx)\gamma_{x,Tx},\gamma_{x,y}) + g_{y}(\rho(y,Ty)\gamma_{y,Ty}),\gamma_{y,x})] \\
&\leq -2[g_{x}(\rho(x,Tx)\gamma_{x,Tx},\rho(x,y)\gamma_{x,y}) + g_{y}(\rho(y,Ty)\gamma_{y,Ty}),\rho(x,y)\gamma_{y,x})] \\
&\leq -2[\product{\vec{xTx}}{\vec{xy}} + \product{\vec{yTy}}{\vec{yx}}] \\
&= \rho^{2}(y,Tx) + \rho^{2}(x,Ty) - \rho^{2}(x,Tx) - \rho^{2}(y,Ty) -2 \rho^{2}(x,y) \\
&\leq 2\rho(x,y)\rho(Tx,Ty) - 2\rho^{2}(x,y) \leq 0.
\end{align*}
Therefore, $A_{T}$ is monotone.
\end{proof}

By using the complementary vector field defined above, we can construct an important device called the Yosida approximation.
\begin{dfn}
Suppose that $H$ has the geodesic extension property, $A : H \multimap TH$ is a monotone vector field with surjectivity condition. Let $\lambda > 0$. The \emph{$\lambda$-Yosida approximation} of $A$, denoted by $A_{\lambda}$, is defined by
\[
A_{\lambda}x := \lambda^{-1}A_{J_{\lambda}}x, \quad \forall x \in H.
\]
\end{dfn}
One may see in the following proposition that Yosida approximations can produce a useful estimate for each $x \in \dom A$.
\begin{prop}\label{prop:HeartWarmingIneq}
Suppose that $H$ has the geodesic extension property and $A$ is a monotone vector field with the surjectivity condition. Then, the inequality
\[
\norm{A_{\lambda}x} \leq \abs{Ax} := \inf_{x^{\ast} \in Ax} \norm{x^{\ast}}
\]
holds for every $\lambda > 0$ and $x \in H$, where the convention $\inf \emptyset = +\infty$.
\end{prop}
\begin{proof}
If $x \not\in \dom A$, the inequality always holds. Suppose now that $x \in \dom A$. If $x \in A^{-1}\0$, then $\norm{A_{\lambda}x} = \norm{0_{x}} = 0 = \abs{Ax}$. On the other hand, if $x \not\in A^{-1}\0$ we pick $u^{\ast} \in Ax$ arbitrarily. For some $t > 0$ and $u \in H\setminus\{x\}$, we have $u^{\ast} = t\rho(x,u)\gamma_{x,u}$ and thus $\norm{u^{\ast}} = t\rho(x,u)$.

Since $H$ has the geodesic extension property, there exists a map $\gamma : \R \to H$ such that $\rho(\gamma(s),\gamma(t)) = \rho(x,u)\abs{s-t}$ for all $s,t \in \R$ with $\gamma(0) = x$ and $\gamma(1) = u$. Let $\lambda > 0$. Put $v := \gamma(\lambda t)$ and $v^{\ast} := \lambda^{-1}\rho(x,v)\gamma_{x,v}$. Observe that $\angle_{x}(\gamma_{x,u},\gamma_{x,v}) = 0$ and that
\[
\norm{v^{\ast}} = \lambda^{-1}\rho(x,v) = \lambda^{-1}\lambda t \rho(x,u) = t\rho(x,u) = \norm{u^{\ast}}.
\]
By the definition of a tangent space, we have $\lambda^{-1}\rho(x,v)\gamma_{x,v} = v^{\ast} = u^{\ast} \in Ax$. This means $x = J_{\lambda}v$. As a consequence, we obtain
\[
\norm{A_{\lambda}x} = \lambda^{-1}\rho(x,J_{\lambda}x) = \lambda^{-1}\rho(J_{\lambda}v,J_{\lambda}x) \leq \lambda^{-1}\rho(v,x) = \norm{v^{\ast}} = \norm{u^{\ast}}.
\]
Since $u^{\ast} \in Ax$ is chosen arbitrarily, we finally have $\norm{A_{\lambda}x} \leq \abs{Ax}$.
\end{proof}

\begin{prop}\label{prop:convexity}
If $H$ satisfies the geodesic extension property and $A : H \multimap TH$ is a monotone vector field with the surjectivity condition, then $\overline{\dom A}$ is convex.
\end{prop}
\begin{proof}
First, we make a claim that $\lim_{\lambda \tendsto 0^{+}} J_{\lambda}x = x$ for any $x \in \overline{\dom A}$. For the moment, consider $x \in \dom A$. For each $\lambda > 0$, we have
\[
\rho(x,J_{\lambda}x) = \lambda\norm{A_{\lambda}x} \leq \lambda \abs{Ax} < +\infty.
\]
Letting $\lambda \tendsto 0^{+}$, we get $\lim_{\lambda \tendsto 0^{+}} J_{\lambda}x = x$. Next, suppose that $x \in \overline{\dom A}$ and let $(x^{k})$ be a sequence in $\dom A$ which converges to $x$. Since $J_{\lambda}$ is nonexpansive, we have
\begin{align*}
\rho(x,J_{\lambda}x) &\leq \rho(x,x^{k}) + \rho(x^{k},J_{\lambda}x^{k}) + \rho(J_{\lambda}x^{k},J_{\lambda}x) \\
&\leq 2\rho(x,x^{k}) + \rho(x^{k},J_{\lambda}x^{k})
\end{align*}
and hence
\[
\limsup_{\lambda \tendsto 0^{+}} \rho(x,J_{\lambda}x)  \leq \limsup_{\lambda \tendsto 0^{+}} [2\rho(x,x^{k}) + \rho(x^{k},J_{\lambda}x^{k})] = 2\rho(x,x^{k}).
\]
Letting $k \tendsto \infty$ on the right hand side then proves the claim.

Put $Q := \{x \in H \,|\, \lim_{\lambda \tendsto 0^{+}} J_{\lambda}x = x\}$. Sicne $J_{\lambda} x \in \dom A$ for all $\lambda > 0$ and $x \in Q$, its limit $x$ is included in $\overline{\dom A}$. Hence $Q \subset \overline{\dom A}$. On the other hand, we have $\overline{\dom A} \subset Q$. We thus have $Q = \overline{\dom A}$.

Now, take $x,y \in Q$ and $t \in (0,1)$. Put $z := (1-t)x \oplus ty$. For any $\lambda > 0$, we obtain
\begin{equation}\label{eqn:est-clD(A)}
\rho(J_{\lambda}z,x) \leq \rho(J_{\lambda}z,J_{\lambda}x) + \rho(J_{\lambda}x,x) \leq \rho(x,z) + \rho(x,J_{\lambda}x) = t\rho(x,y) + \rho(x,J_{\lambda}x),
\end{equation}
and similarly have
\begin{equation}\label{eqn:est-clD(A)2}
\rho(J_{\lambda}z,y) \leq \rho(J_{\lambda}z,J_{\lambda}y) + \rho(J_{\lambda}y,y) \leq \rho(y,z) + \rho(y,J_{\lambda}y) = (1-t)\rho(x,y) + \rho(y,J_{\lambda}y).
\end{equation}
Take any sequence $(\lambda_{k})$ in $(0,\infty)$ such that $\lambda_{k} \tendsto 0$. Then it follows from \eqref{eqn:est-clD(A)} and $J_{\lambda_{k}}x \tendsto x$ that the sequence $(J_{\lambda_{k}}z)$ is bounded. So, $(J_{\lambda_{k}}z)$ contains a $\Delta$-convergent subsequence $(J_{\lambda_{k_{i}}}z)$ with $\Delta$-limit $p \in H$. Putting $\lambda = \lambda_{k_{i}}$ in \eqref{eqn:est-clD(A)} and \eqref{eqn:est-clD(A)2} and letting $i \tendsto \infty$, we obtain
\begin{equation}\label{eqn:liminf1}
\liminf_{i \tendsto \infty} \rho(J_{\lambda_{k_{i}}}z,x) \leq t\rho(x,y)
\end{equation}
and also
\begin{equation}\label{eqn:liminf2}
\liminf_{i \tendsto \infty} \rho(J_{\lambda_{k_{i}}}z,y) \leq (1-t)\rho(x,y).
\end{equation}
By \eqref{eqn:CN}, the $\Delta$-lower semicontinuity of $\rho(x,\cdot)$, and ineqalities \eqref{eqn:liminf1} as well as \eqref{eqn:liminf2}, we have
\begin{align*}
\rho^{2}(p,z) &\leq (1-t)\rho^{2}(x,p) + t\rho^{2}(y,p) - t(1-t)\rho^{2}(x,y) \\
& \leq (1-t)\left[\liminf_{i \tendsto \infty} \rho(x,J_{\lambda_{k_{i}}}z)\right]^{2} + t\left[\liminf_{i \tendsto \infty} \rho(y,J_{\lambda_{k_{i}}}z)\right]^{2} - t(1-t)\rho^{2}(x,y) \\
& \leq (1-t)t^{2}\rho^{2}(x,y) + t(1-t)^{2}\rho^{2}(x,y) - t(1-t)\rho^{2}(x,y) \\
& = 0.
\end{align*}
Hence $(J_{\lambda_{k_{i}}}z)$ is $\Delta$-convergent to $z$, which implies that $\omega_{\Delta}(J_{\lambda_{k}}z) = \{z\}$. By Proposition \ref{prop:singletonOmegaDelta} we conclude that $(J_{\lambda_{k}}z)$ is $\Delta$-convergent to $z$.

Next, observe that
\[
\rho(J_{\lambda_{k}}z,x) \leq \rho(J_{\lambda_{k}}z,J_{\lambda_{k}}x) + \rho(J_{\lambda_{k}}x,x) \leq \rho(z,x) + \rho(J_{\lambda_{k}}x,x).
\]
This inequality together with the $\Delta$-lower semicontinuity of $\rho(\cdot,x)$, we obtain
\[
\rho(z,x) \leq \liminf_{k \tendsto \infty} \rho(J_{\lambda_{k}}z,x) \leq \limsup_{k \tendsto \infty} \rho(J_{\lambda_{k}}z,x) \leq \rho(z,x),
\]
and so $\lim_{k} \rho(J_{\lambda_{k}}z,x) = \rho(z,x)$. By the Kadec-Klee property (Lemma \ref{lem:Kadec-Klee}), we get $\lim_{k} J_{\lambda_{k}}z = z$. Since $(\lambda_{k})$ is an arbitrary sequence of positive numbers with $\lambda_{k} \tendsto 0$, we can conclude that $\lim_{\lambda \tendsto 0^{+}} J_{\lambda}z = z$ and so $z \in Q$. This shows the convexity of $Q = \overline{\dom A}$.
\end{proof}

\begin{thm}\label{thm:Asymp0}
Suppose that $H$ satisfies the geodesic extension property, and $A : H \multimap TH$ is a monotone vector field with the surjectivity condition. Then, 
\[
\lim_{\lambda \tendsto 0^{+}} J_{\lambda}x = P_{\overline{\dom A}} x,
\]
for every $x \in H$.
\end{thm}
\begin{proof}
Put $P := P_{\overline{\dom A}}$. Let $x \in H$ and $\lambda > 0$. By the monotonicity of $A$, we have
\begin{align}
0 &\geq 2\lambda\rho(J_{\lambda}x,J_{\lambda}Px) \left[ g_{J_{\lambda}x}(\lambda^{-1}\rho(J_{\lambda}x,x)\gamma_{J_{\lambda}x,x},\gamma_{J_{\lambda}x,J_{\lambda}Px}) \right. \nonumber\\
&\qquad\qquad \qquad \qquad\qquad  + \left. g_{J_{\lambda}Px}(\lambda^{-1}\rho(J_{\lambda}Px,Px)\gamma_{J_{\lambda}Px,Px},\gamma_{J_{\lambda}Px,J_{\lambda}x}) \right] \nonumber\\
&= 2 \left[ g_{J_{\lambda}x}(\rho(J_{\lambda}x,x)\gamma_{J_{\lambda}x,x},\rho(J_{\lambda}x,J_{\lambda}Px)\gamma_{J_{\lambda}x,J_{\lambda}Px}) \right. \nonumber\\
&\qquad\qquad\qquad + \left. g_{J_{\lambda}Px}(\rho(J_{\lambda}Px,Px)\gamma_{J_{\lambda}Px,Px},\rho(J_{\lambda}x,J_{\lambda}Px) \gamma_{J_{\lambda}Px,J_{\lambda}x}) \right] \nonumber\\
&\geq 2 \left[\product{\gamma_{J_{\lambda}x,x}}{\gamma_{J_{\lambda}x,J_{\lambda}Px}} + \product{\gamma_{J_{\lambda}Px,Px}}{\gamma_{J_{\lambda}Px,J_{\lambda}x}}\right] \nonumber\\
& = \rho^{2}(x,J_{\lambda}x) + 2\rho^{2}(J_{\lambda}x,J_{\lambda}Px) - \rho^{2}(x,J_{\lambda}Px) + \rho^{2}(Px,J_{\lambda}Px) - \rho^{2}(Px,J_{\lambda}x). \label{eqn:est-zero-1}
\end{align}
We recall from \ref{cdn:ProjIneq} in Theorem \ref{thm:projection} that
\[
\rho^{2}(x,J_{\lambda}x) \geq \rho^{2}(x,Px) + \rho^{2}(Px,J_{\lambda}x).
\]
Substitute this inequality in \eqref{eqn:est-zero-1}, we get
\begin{align*}
0 &\geq \left(\rho^{2}(x,Px) + \rho^{2}(Px,J_{\lambda}x)\right) + 2\rho^{2}(J_{\lambda}x,J_{\lambda}Px) - \rho^{2}(x,J_{\lambda}Px) \\
& \qquad + \rho^{2}(Px,J_{\lambda}Px) - \rho^{2}(Px,J_{\lambda}x) \\
& = \rho^{2}(x,Px) + 2\rho^{2}(J_{\lambda}x,J_{\lambda}Px) - \rho^{2}(x,J_{\lambda}Px) + \rho^{2}(Px,J_{\lambda}Px).
\end{align*}
From the proof of the previous proposition and the fact that $Px \in \overline{\dom A}$, letting $\lambda \tendsto 0^{+}$ in the above inequality yields
\begin{align*}
0 & \geq \rho^{2}(x,Px) + 2\limsup_{\lambda \tendsto 0^{+}} \rho^{2}(J_{\lambda}x,J_{\lambda}Px) - \lim_{\lambda \tendsto 0^{+}} \rho^{2}(x,J_{\lambda}Px) + \lim_{\lambda \tendsto 0^{+}} \rho^{2}(Px,J_{\lambda}Px) \\
& = \rho^{2}(x,Px) + 2\limsup_{\lambda \tendsto 0^{+}} \rho^{2}(J_{\lambda}x,J_{\lambda}Px) - \rho^{2}(x,Px) + \rho^{2}(Px,Px) \\
& = 2\limsup_{\lambda \tendsto 0^{+}} \rho^{2}(J_{\lambda}x,J_{\lambda}Px) \geq 2\liminf_{\lambda \tendsto 0^{+}} \rho^{2}(J_{\lambda}x,J_{\lambda}Px) \geq 0,
\end{align*}
and so $\lim_{\lambda \tendsto 0^{+}} \rho^{2}(J_{\lambda}x,J_{\lambda}Px) = 0$. This equality and the triangle inequality further imply
\begin{align*}
0 &= \lim_{\lambda \tendsto 0^{+}} \rho(J_{\lambda}x,J_{\lambda}Px) \geq \limsup_{\lambda \tendsto 0^{+}} \rho(J_{\lambda}x,Px) - \lim_{\lambda \tendsto 0^{+}} \rho(Px,J_{\lambda}Px) \\
&\geq \limsup_{\lambda \tendsto 0^{+}} \rho(J_{\lambda}x,Px) \geq \liminf_{\lambda \tendsto 0^{+}} \rho(J_{\lambda}x,Px) \geq 0.
\end{align*}
We finally conclude that $\lim_{\lambda \tendsto 0^{+}}J_{\lambda}x = Px$.
\end{proof}

\begin{prop}
If $H$ satisfies the geodesic extension property, then the mapping $\lambda \mapsto J_{\lambda} x$ is continuous on $(0,+\infty)$ for every $x \in H$. If $x \in \overline{\dom A}$, then it is continuous on $[0,+\infty)$.
\end{prop}
\begin{proof}
Let $z \in \dom A$ and $\lambda \geq \mu > 0$. By \ref{cdn:J-NX} and \ref{cdn:J-ResolventIden} in Proposition \ref{prop:properties-J} and Proposition \ref{prop:HeartWarmingIneq}, we get
\begin{align}\label{eqn:est-cont-J}
\rho(J_{\mu}z,J_{\lambda}z) &\leq \rho\left(J_{\mu}z,J_{\mu}\left(\left(1-\frac{\mu}{\lambda}\right)J_{\lambda}z \oplus \frac{\mu}{\lambda}z\right)\right) \nonumber\\
&\leq \rho\left(z,\left(1-\frac{\mu}{\lambda}\right)J_{\lambda}z \oplus \frac{\mu}{\lambda}z\right) \nonumber\\
& = \left(1 - \frac{\mu}{\lambda}\right)\rho(z,J_{\lambda}z) \leq (\lambda - \mu)\abs{Az}.
\end{align}
Therefore, $\lambda \mapsto J_{\lambda}z$ is continuous on $(0,\infty)$ for $z \in \dom A$.
Now, let $x \in H$ and $I := [a,b]$ be a bounded closed subinterval of $(0,+\infty)$. We have
\[
\rho(x,J_{t}x) \leq \rho(x,J_{t}z) + \rho(J_{t}z,J_{t}x) \leq \rho(x,J_{t}z) + \rho(z,x)
\]
for all $t \in I$. By the continuity of $J_{t}z$, we know that $\rho(x,J_{t}z)$ is bounded on $I$. The above inequality implies that $\rho(x,J_{t}x)$ is also bounded on $I$. Thus there exists $M > 0$ such that we may similarly obtain
\[
\rho(J_{\mu}x,J_{\lambda}x) \leq \left(1 - \frac{\mu}{\lambda}\right)\rho(x,J_{\lambda}x) \leq \left(1 - \frac{\mu}{\lambda}\right)M
\]
whenever $\mu,\lambda \in (a,b)$ and $\mu \leq \lambda$. This proves the desired continuity. Moreover, Theorem \ref{thm:Asymp0} shows the right continuity at $\lambda = 0$ for $x \in \overline{\dom A}$.
\end{proof}

\begin{thm}\label{thm:asymp_infty}
Suppose that $H$ satisfies the geodesic extension property, $A : H \multimap TH$ is a monotone vector field with the surjectivity condition, and $A^{-1}(\0) \neq \emptyset$. Then,
\[
\lim_{\lambda \tendsto \infty} J_{\lambda}x = P_{A^{-1}(\0)} x,
\]
for every $x \in H$.
\end{thm}
\begin{proof}
Let $x \in H$ and $(\lambda_{k}) \subset (0,\infty)$ be a sequence with $\lambda_{k} \tendsto \infty$. For convenience, we put $u := P_{A^{-1}(\0)} x$ and $x^{k} := J_{\lambda_{k}} x$ for each $k \in \N$. In view of the inequality \eqref{eqn:firmlynonspreading}, we have for each $k \in \N$:
\begin{equation}\label{eqn:bound-1}
\rho^{2}(x^{k},x) \leq \rho^{2}(x^{k},x) + \rho^{2}(x^{k},u) \leq \rho^{2}(x,u) =: K.
\end{equation}
Therefore, $(x^{k})$ is bounded and so it admits a subsequence $(x^{k_{i}})$ which is $\Delta$-convergent to some point $z \in H$. Recall that $\lambda_{k_{i}}^{-1}\rho(x^{k_{i}},x)\gamma_{x^{k_{i}},x} \in Ax^{k_{i}}$, and so
\[
\abs{Ax^{k_{i}}} \leq \norm{\lambda_{k_{i}}^{-1}\rho(x^{k_{i}},x)\gamma_{x^{k_{i}},x}} = \lambda_{k_{i}}^{-1}\rho(x^{k_{i}},x) \leq \lambda_{k_{i}}^{-1}K.
\]
Fix any $\mu > 0$. From Proposition \ref{prop:HeartWarmingIneq}, we have
\[
\mu^{-1}\rho(x^{k_{i}},J_{\mu}x^{k_{i}}) = \norm{A_{\mu}x^{k_{i}}} \leq \abs{Ax^{k_{i}}} \leq \lambda_{k_{i}}^{-1}K.
\]
It follows that
\[
\rho(x^{k_{i}},J_{\mu}x^{k_{i}}) \leq \mu\lambda_{k_{i}}^{-1}K,
\]
which yields $\lim_{k}\rho(x^{k_{i}},J_{\mu}x^{k_{i}}) = 0$. By the Demiclosedness Principle (Theorem \ref{thm:demiclosed}), we have $z \in \Fix(J_{\mu}) = A^{-1}\0$. Since $\rho^{2}(\cdot,x)$ is $\Delta$-lower semicontinuous and $(x^{k_{i}})$ is $\Delta$-convergent to $z$, we obtain from \eqref{eqn:bound-1} the following:
\[
\rho^{2}(u,x) \leq \rho^{2}(z,x) \leq \liminf_{i \tendsto \infty} \rho^{2}(x^{k_{i}},x) \leq \limsup_{i \tendsto \infty} \rho^{2}(x^{k_{i}},x) \leq \rho^{2}(u,x).
\]
Therefore, $z = u$ and so $\omega_{\Delta}(x^{k}) = \{u\}$. By Proposition \ref{prop:singletonOmegaDelta} we conclude that $(x^{k})$ is $\Delta$-convergent to $u$. With similar procedure above, we obtain $\lim_{k} \rho(x^{k},x) = \rho(u,x)$. By the Kadec-Klee property (Lemma \ref{lem:Kadec-Klee}), we have the strong convergence $\lim_{k} J_{\lambda_{k}}x = Px$. Since $(\lambda_{k})$ is an arbitrary sequence with $\lambda_{k} \tendsto \infty$, the convergence is attained.
\end{proof}

\section{The generation of a nonexpansive semigroup}

In this section, we state and prove the generation theorem applied to a monotone vector field on a complete $\CAT(0)$ space. Our results in this section extend several results in the literature. For instance, the generation theorems for proper, convex, lower semicontinuous functions in complete $\CAT(0)$ spaces of Jost \cite[Theorem 1.3.13]{MR1652278}, Mayer \cite[Theorem 1.13]{MR1651416}, Stojkovic \cite[Theorem 2.18]{MR2879568}, and Ba\v{c}\'{a}k \cite[Theorem 1.5]{MR3047087} are successfully amplified to use with monotone vector fields. Moreover, the results of Iwamiya and Okochi \cite[Theorem 4.1]{MR1979730} are lifted to a greater generality where the previous requirement for the smoothness of class $C^{3}$ is completely removed.

Recall that for a nonempty subset $C \subset H$, a family $\{S(t)\}_{t \geq 0}$ is said to be a \emph{nonexpansive semigroup} on $C$ if the following conditions are satisfied:
\begin{enumerate}[label=(\roman*)]
\item $S(t) : C \to C$ is nonexpansive for each $t \geq 0$.
\item $S(s+t) = S(s) \circ S(t)$ for all $s,t \geq 0$.
\item $S(0) = I$, where $I$ denotes the identity mapping on $C$.
\item $t \mapsto S(t)x$ is continuous on $[0,+\infty)$ for each $x \in C$.
\end{enumerate}

The following technical lemma can be extracted from the proof of \cite[Theorem 1 of Chapter XIV.7]{MR1336382}. It also appears in a Japanese textbook \cite{MR840217} in the present form. Here, we adopt the notation $\N_{0} := \N \cup \{0\}$. \begin{lem}[\cite{MR840217,MR1336382}]\label{lem:doubleSeq}
Let $\lambda > 0$. For each $k \in \N_{0}$, let $\mu_{k} \in (0,\lambda]$, $\alpha_{k} := \frac{\mu_{k}}{\lambda}$, $\beta_{k} := 1 - \alpha_{k}$, and
\[
t_{k} := \left\{
\begin{array}{ll}
0	&\text{if $k = 0$;} \medskip\\
\sum_{i=1}^{k} \mu_{i}	&\text{if $k \in \N$.}
\end{array}
\right.
\]
Let $(A_{j,k})$ be a double sequence such that
\[
\left\{
\begin{array}{l}
A_{0,k} \leq k\lambda, \quad  A_{j,0} \leq t_{j} \quad (\forall j,k \in \N_{0}); \smallskip\\
A_{j,k} \leq \alpha_{j}A_{j-1,k-1} + \beta_{j}A_{j-1,k} \quad (\forall j,k \in \N).
\end{array}
\right.
\]
Then the following inequality holds:
\[
A_{j,k} \leq \sqrt{(k\lambda - t_{j})^{2} + k\lambda^{2}} + \sqrt{(k\lambda - t_{j})^{2} + \lambda t_{j}} \quad (\forall j,k \in \N_{0}).
\]
\end{lem}

Now, we are ready to prove the main result of this section. The theorem extends a similar result of Crandal and Liggett \cite{MR0287357} to the setting of a complete $\CAT(0)$ space.

\begin{thm}\label{thm:S(t)x}
Suppose that $H$ satisfies the geodesic extension property, $A : H \multimap TH$ is a monotone vector field with the surjectivity condition. Assume that $x \in \overline{\dom A}$. Then, the following limit exists and is uniform on each bounded subinterval of $[0,+\infty)$:
\begin{equation}\label{eqn:S(t)x}
S(t)x := \lim_{k \tendsto \infty} J_{t/k}^{k} x \quad (\forall t \in [0,+\infty)).
\end{equation}
Moreover, the following estimate holds for any $x \in \dom A$:
\begin{equation}\label{eqn:semigroupestimate}
\rho(J_{t/k}^{k} x,S(t)x) \leq \abs{Ax} \cdot \frac{2t}{\sqrt{k}}
\end{equation}
for any $t \geq 0$ and $k \in \N$.
\end{thm}
\begin{proof}
Let $t \geq 0$ be given. Then we first show the convergence on \eqref{eqn:S(t)x} by showing that $(J_{t/k}^{k})$ is Cauchy. Let $x \in \dom A$ and $\lambda > 0$ be given. Let $(\mu_{k})$ and $(t_{k})$ be sequences given in Lemma \ref{lem:doubleSeq}. Put
\[
\tilde A_{j,k} := \left\{
\begin{array}{ll}
\rho(J_{\mu_{j}}J_{\mu_{j-1}}\cdots J_{\mu_{1}}x,J_{\lambda}^{k}x)	& (\forall j \in \N, \forall k \in \N_{0}); \medskip\\
\rho(x,J_{\lambda}^{k}x)	& (j=0, \forall k \in \N_{0}).
\end{array}
\right.
\]
For $k \in \N_{0}$, it follows from Propositions \ref{prop:properties-J} and \ref{prop:HeartWarmingIneq} that
\begin{align*}
\tilde A_{0,k} &= \rho(x,J_{\lambda}^{k}x) \\
&\leq \rho(x,J_{\lambda}x) + \rho(J_{\lambda}x,J_{\lambda}^{2}x) + \cdots + \rho(J_{\lambda}^{k-1}x,J_{\lambda}^{k}x) \\
&\leq k \rho(x,J_{\lambda}x) = k\lambda \norm{A_{\lambda}x} \leq k\lambda\abs{Ax}.
\end{align*}
For $j \in \N$, we have
\begin{align*}
\tilde A_{j,0} &= \rho(J_{\mu_{j}}J_{\mu_{j-1}}\cdots J_{\mu_{1}}x,x) \\
&\leq \rho(J_{\mu_{j}}J_{\mu_{j-1}}\cdots J_{\mu_{1}}x,J_{\mu_{j}}J_{\mu_{j-1}}\cdots J_{\mu_{2}}x) + \rho(J_{\mu_{j}}J_{\mu_{j-1}}\cdots J_{\mu_{2}}x,J_{\mu_{j}}J_{\mu_{j-1}}\cdots J_{\mu_{3}}x) \\
&\quad + \cdots + \rho(J_{\mu_{j}}J_{\mu_{j-1}}x,J_{\mu_{j}}x) + \rho(J_{\mu_{j}}x,x) \\
&\leq \rho(J_{\mu_{1}}x,x) + \rho(J_{\mu_{2}}x,x) + \cdots + \rho(J_{\mu_{j}}x,x) \\
&= \mu_{1}\norm{A_{\mu_{1}}x} + \mu_{2}\norm{A_{\mu_{2}}x} + \cdots + \mu_{j}\norm{A_{\mu_{j}}x} \\
&\leq (\mu_{1} + \cdots + \mu_{j})\abs{Ax} = t_{j}\abs{Ax}.
\end{align*}
Now, let $j,k \in \N$. By using \ref{cdn:J-NX} and \ref{cdn:J-ResolventIden} of Proposition \ref{prop:properties-J}, we get
\begin{align*}
\tilde A_{j,k} &= \rho(J_{\mu_{j}}J_{\mu_{j-1}}\cdots J_{\mu_{1}}x,J_{\lambda}^{k}x) \\
&= \rho\left(J_{\mu_{j}}J_{\mu_{j-1}}\cdots J_{\mu_{1}}x,J_{\mu_{j}}\left(\frac{\mu_{j}}{\lambda}J_{\lambda}^{k-1}x \oplus \left(1 - \frac{\mu_{k}}{\lambda}\right)J_{\lambda}^{k}x\right)\right) \\
&\leq \rho\left(J_{\mu_{j-1}}\cdots J_{\mu_{1}}x,\frac{\mu_{j}}{\lambda}J_{\lambda}^{k-1}x \oplus \left(1 - \frac{\mu_{k}}{\lambda}\right)J_{\lambda}^{k}x\right) \\
&\leq \frac{\mu_{j}}{\lambda} \rho(J_{\mu_{j-1}}\cdots J_{\mu_{1}}x,J_{\lambda}^{k-1}x) + \left(1 - \frac{\mu_{k}}{\lambda}\right)\rho(J_{\mu_{j-1}}\cdots J_{\mu_{1}}x,J_{\lambda}^{k}x) \\
&= \alpha_{j}\tilde A_{j-1,k-1} + \beta_{j}\tilde A_{j-1,k}.
\end{align*}
If $\abs{Ax} = 0$, then $\rho(x,J_{\lambda}^{k}x) = \tilde A_{0,k} = 0$ and so $(J_{t/k}^{k} x)$ is a constant sequence. If $\abs{Ax} > 0$, we put $A_{j,k} := \tilde A_{j,k}/\abs{Ax}$ for all $j,k \in \N_{0}$. Then, $(A_{j,k})$ satisfies the hypotheses of Lemma \ref{lem:doubleSeq} so that we have
\[
\tilde A_{j,k} \leq \abs{Ax} \left(\sqrt{(k\lambda - t_{j})^{2} + k\lambda^{2}} + \sqrt{(k\lambda - t_{j})^{2} + \lambda t_{j}}\right) \quad (\forall j,k \in \N_{0}).
\]
Fix $t > 0$ and $n,m \in \N$ with $m < n$. Take $\lambda := \frac{t}{m}$ and for each $k \in \N_{0}$, take $\mu_{k} := \frac{t}{n}$. So, we have $t_{0} = 0$ and $t_{n} = t$. Hence, we have
\begin{align*}
\rho(J_{t/n}^{n}x,J_{t/m}^{m}x) &= \rho(J_{\mu_{n}}J_{\mu_{n-1}}\cdots J_{\mu_{1}}x,J_{\lambda}^{m}x) = \tilde A_{n,m} \\
&\leq \abs{Ax}\left(\sqrt{(m\lambda - t_{n})^{2} + m\lambda^{2}} + \sqrt{(m\lambda - t_{n})^{2} + \lambda t_{n}}\right) \\
& = \abs{Ax}\cdot\frac{2t}{\sqrt{m}}.
\end{align*}
This shows that $(J_{t/k}^{k}x)$ is Cauchy and so the limit $S(t)x$ is defined, by the completeness of $H$. The above inequality also yields the estimate \eqref{eqn:semigroupestimate}. Take any $T > 0$, we further have
\[
\sup_{t \in [0,T]}\rho(J_{t/k}^{k} x,S(t)x) \leq \abs{Ax} \cdot \frac{2T}{\sqrt{k}} \quad (\forall k \in \N),
\]
which guarantee that the convergence is uniform on compact intervals of $[0,\infty)$.

Next, we show convergence for any $x \in \overline{\dom A}$. For instance, let $(x^{i})$ be a sequence in $\dom A$ such that $x^{i} \tendsto x$. For any $m,m \in \N$, we have
\begin{align*}
\rho(J_{t/n}^{n}x,J_{t/m}^{m}x) &\leq \rho(J_{t/n}^{n}x,J_{t/n}^{n}x^{i}) + \rho(J_{t/n}^{n}x^{i},J_{t/m}^{m}x^{i}) + \rho(J_{t/m}^{m}x^{i},J_{t/m}^{m}x) \\
&\leq 2\rho(x^{i},x) + \rho(J_{t/n}^{n}x^{i},J_{t/m}^{m}x^{i}).
\end{align*}
Since $x^{i} \tendsto x$ and the sequence $(J_{t/k}^{k}x^{i})_{n \in \N}$ is Cauchy for any $i \in \N$, we conclude that $(J_{t/k}^{k}x)$ is a Cauchy sequence. The convergence is again obtained by the completeness of $H$.

It remains to show that $S(t)x$ is attained uniformly on each bounded interval $[0,T]$ for $x \in \overline{\dom A}$. Again, let $(x^{i})$ be a sequence in $\dom A$ convergent to $x$. Since each mapping $J_{t}$ is nonexpansive for any $t \geq 0$, it follows that $S(t)$ is also nonexpansive for each $t \geq 0$. Consider for $0 \leq t \leq T$, we have
\begin{align*}
\rho(J_{t/k}^{k}x,S(t)x) &\leq \rho(J_{t/k}^{k}x,J_{t/k}^{k}x^{i}) + \rho(J_{t/k}^{k}x^{i},S(t)x^{i}) + \rho(S(t)x^{i},S(t)x) \\
&\leq 2\rho(x,x^{i}) + \rho(J_{t/n}^{n}x^{k},S(t)x^{k}) \\
&\leq 2\rho(x,x^{i}) + \sup_{0 \leq t \leq T} \rho(J_{t/k}^{k}x^{i},S(t)x^{i}) \\
&\leq 2\rho(x,x^{i}) + \abs{Ax^{i}}\cdot\frac{2T}{\sqrt{k}}.
\end{align*}
Letting $\varepsilon > 0$ be arbitrary, we can choose $i_{0} \in \N$ such that $2\rho(x,x^{i_{0}}) < \varepsilon$. This further implies
\[
\sup_{0 \leq t \leq T} \rho(J_{t/k}^{k}x,S(t)x) \leq \varepsilon + \abs{Ax^{i_{0}}}\cdot\frac{2T}{\sqrt{k}}
\]
and hence
\[
\lim_{k \tendsto \infty} \sup_{0 \leq t \leq T} \rho(J_{t/k}^{k}x,S(t)x) \leq \varepsilon.
\] 
Since $\varepsilon$ is chosen arbitrarily, then convergence is uniform on $[0,T]$.
\end{proof}

In the above theorem, we have in fact defined a family $\{S(t)\}_{t \geq 0}$ of mappings $S(t) : \overline{\dom A} \to \overline{\dom A}$. We next show that this family is a nonexpansive semigroup. Note that the semigroup constructed in this way is said to be \emph{generated by $A$}.
\begin{thm}
Under the assumptions of Theorem \ref{thm:S(t)x}, the family $\{S(t)\}_{t \geq 0}$ generated by the formula \eqref{eqn:S(t)x} is a nonexpansive semigroup.
\end{thm}
\begin{proof}
The fact that $S(0) = I$, $S(t)$ is nonexpansive and $t \mapsto S(t)x$ is continuous on $[0,\infty)$ follows from $J_{0} = I$, $J_{t}$ is nonexpansive, $t \mapsto J_{t}x$ is continuous on $[0,\infty)$ and the uniform convergence of $\{J_{t/k}^{k} x\}$ on each bounded subinterval of $[0,\infty)$ for $x \in \overline{\dom A}$. It is therefore sufficient to show only $S(s+t) = S(s)S(t)$.

Let $t \geq 0$. We first show that $S(nt) = S(t)^{n}$ for all $n \in \N$. For $n =2$, we have
\begin{align*}
\rho(S(t)^{2}x,J_{t/k}^{2k}x) &\leq \rho(S(t)S(t)x,J_{t/k}^{k}S(t)x) + \rho(J_{t/k}^{k}S(t)x,J_{t/k}^{k}J_{t/k}^{k}x) \\
&\leq \rho(S(t)^{2}x,J_{t/k}^{k}S(t)x) + \rho(S(t)x,J_{t/k}^{k}x).
\end{align*}
Letting $k \tendsto \infty$, we have $S(t)^{2}x = \lim_{n} J_{t/k}^{2k} x = \lim_{n} J_{2t/2k}^{2k} x = S(2t)x$. Next, let $n \geq 3$ and suppose that $S((n-1)t) = S(t)^{n-1}$ holds true. We have
\[
\rho(S(t)^{n}x,J_{t/k}^{nk}x) \leq \rho(S(t)S(t)^{n-1}x,J_{t/k}^{k}S(t)^{n-1}x) + \rho(S(t)^{n-1}x,J_{t/k}^{(n-1)k}x).
\]
Again, letting $k \tendsto \infty$, we similarly have $S(t)^{n} = S(nt)$. The claim that $S(nt) = S(t)^{n}$ for all $n \in \N$ is thus proved.

Now, for two positive rationals $s := c/d$ and $t := p/q$ with $c,d,p,q \in \N$, we have
\begin{align*}
S(s+t) & = S\left(\frac{c}{d} + \frac{p}{q}\right) = S\left(\frac{cq + dp}{dq}\right) = S\left(\frac{1}{dq}\right)^{cq + dp} \\
& = S\left(\frac{1}{dq}\right)^{cq}S\left(\frac{1}{dq}\right)^{dp} = S\left(\frac{c}{d}\right)S\left(\frac{p}{q}\right) = S(s)S(t).
\end{align*}
The continuity of $t \mapsto S(t)$ at each $x \in \overline{\dom A}$ implies that $S(s+t) = S(s)S(t)$ for $s,t \geq 0$.
\end{proof}

The trajectory $(S(t)x)_{t \geq 0}$ can behave very unstably. However, if a stationary point exists, the trajectory is bounded. As is known from Hilbert space theory, $(S(t)x)_{t \geq 0}$ is not necessarily weakly convergent at all. The following result is our final result, and it shows that the mentioned trajectory is bounded if a stationary point exists and if all the accumulated points are contained in $A^{-1}\0$, we have its $\Delta$-convergence to a stationary.

\begin{thm}\label{thm:S(t)x-convergent}
Suppose that all the assumptions of Theorem \ref{thm:S(t)x} hold, $A^{-1}\0 \neq \emptyset$, and $\{S(t)\}_{t \geq 0}$ be the semigroup generated by $A$. Let $x \in \overline{\dom A}$. Then, $(S(t)x)_{t \geq 0}$ is bounded. Moreover, if all $\Delta$-accumulated points of $(S(t)x)_{t \geq 0}$ is contained in $A^{-1}\0$, then $(S(t)x)_{t \geq 0}$ is $\Delta$-convergent to a stationary point as $t \tendsto \infty$.
\end{thm}
\begin{proof}
Suppose that $p \in A^{-1}\0$. Then, it is a fixed point to all $S(t)$'s. Therefore, we have
\[
\rho(p,S(t)x) \leq \rho(p,x), \quad \forall t \geq 0.
\]
This shows the boundedness, and it implies that a $\Delta$-accumulation point of $(S(t)x)_{t \geq 0}$ exists. Assume that every $\Delta$-accumulation points of $(S(t)x)_{t \geq 0}$ is contained in $A^{-1}\0$. Then, we have
\[
\rho(p,S(t)x) = \rho(S(t-s)p,S(t-s)S(s)x) \leq \rho(p,S(s)x),
\]
for any $0 \leq s \leq t$. Hence, $(S(t)x)_{t \geq 0}$ is Fej\'er monotone with respect to $A^{-1}\0$. By the hypothesis and Proposition \ref{prop:Fejer}, we conclude that $(S(t)x)_{t \geq 0}$ is $\Delta$-convergent to a stationary point of $A$ as $t \tendsto \infty$.
\end{proof}

\section*{Conclusions}

In this paper, we utilized the concept of tangent spaces to develop a theory on monotone vector fields and generalized gradient flows over complete $\CAT(0)$ spaces. This approach is a natural extension of the theory from the frameworks of Hilbert spaces and Hadamard manifolds. Important instruments in our studies include resolvents and Yosida approximations. Among others, we obtained the resolvent identity (see in Proposition \ref{prop:properties-J}), the convexity of $\overline{\dom A}$ for a monotone vector field $A$ (Proposition \ref{prop:convexity}), asymptotic convergence for the resolvents (Theorems \ref{thm:Asymp0} and \ref{thm:asymp_infty}) and derived a useful estimate for the Yosida approximations (Proposition \ref{prop:HeartWarmingIneq}). Finally, we used such devices to establish a generation theorem of nonexpansive semigroup (Theorem \ref{thm:S(t)x}) which improves and generalizes, to some extents, results of Jost \cite[Theorem 1.3.13]{MR1652278}, Mayer \cite[Theorem 1.13]{MR1651416}, Stojkovic \cite[Theorem 2.18]{MR2879568}, Ba\v{c}\'{a}k \cite[Theorem 1.5]{MR3047087}, and Iwamiya and Okochi \cite[Theorem 4.1]{MR1979730}. The error estimate for the generation is also given in terms of Yosida approximations up to any given accuracy and time.

We also propose the following open questions which are yet to be considered from a viewpoint of our paper.
\begin{enumerate}[label=Q\arabic*.]
\item Does the surjectivity condition always hold for a maximally monotone vector field?
\item Is it possible to drop the geodesic extension property in all definitions and results, where the Yosida approximations are involved?
\end{enumerate}


\section*{Acknowledgements}

The second author was supported by JSPS KAKENHI Grant No. 17K05372. The third author was supported by he Thailand Research Fund (TRF) and the King Mongkut’s University of Technology Thonburi (KMUTT) under the TRF Research Scholar Award (Grant No. RSA6080047).

\renewcommand\bibname{References}
\bibliographystyle{abbrv}
\bibliography{vfflow}

\end{document}